\documentclass[a4paper, 12pt]{elsarticle}
\usepackage{amssymb,amsmath,amsthm,bm}
\theoremstyle{definition}
\newtheorem{dfn}{Definition}[section]
\theoremstyle{plain}
\newtheorem{thm}{Theorem}[section]
\newtheorem{lem}{Lemma}[section]
\newtheorem{prop}{Proposition}[section]
\newtheorem{cor}{Corollary}[section]

\theoremstyle{remark}
\newtheorem{rem}{Remark}[section]
\makeatletter

\@addtoreset{equation}{section}
\makeatother
\journal{Journal of Algebra}

\begin{document}

\begin{frontmatter}

\title{The Free Generalized Vertex Algebras and Generalized Principal Subspaces}
\author{Kazuya Kawasetsu}
\address{Department of Mathematical Sciences, University of Tokyo, Komaba, Tokyo, 153-8914, Japan.}

\ead{kawasetu@ms.u-tokyo.ac.jp}

\begin{abstract}
The notion of {\it free} generalized vertex algebras is introduced.
It is equivalent to the notion of {\it generalized principal subspaces} associated with lattices which are
not necessarily integral.
Combinatorial bases and the characters of the free generalized vertex algebras are given.
As an application, the commutants of principal subspaces are described by using generalized principal subspaces.
\end{abstract}

\begin{keyword}
\texttt{generalized vertex algebras, free vertex algebras,  principal subspaces.}
\end{keyword}

\end{frontmatter}

\section{Introduction}

The notion of vertex superalgebras is an analogy of the notion of commutative associative algebras.
They have countably many multiplications with the {\it locality} axioms.
A {\it free} vertex superalgebra $F=F(B,\mathcal{N})$ \cite{R} is the universal object for some vertex superalgebras.
It is described by the set of generators $B$ and system of locality bounds $\mathcal{N}:B\times B\rightarrow \mathbb{Z}$, which determine the form of the locality axiom on $B$.
If a vertex algebra $V$ is generated by $B$ and have the system of locality bounds $\mathcal{N}$ among the elements of $B$,
then there exists the unique projection $F(B,\mathcal{N})\rightarrow V$
which sends $B\subset F$ identically onto $B\subset V$.
Note that $F$ is determined not only by $B$ but also by $\mathcal{N}$.
It was first remarked in \cite{B}.
In \cite{R}, combinatorial bases and the graded dimensions of the free vertex superalgebras were given.

There is a notion which is equivalent to the notion of free vertex superalgebras
--- the {\it principal subspaces} of the lattice vertex superalgebras \cite{MP}.
Principal subspaces are first introduced by Feigin and Stoyanovsky \cite{SF}.
A (Feigin-Stoyanovsky) principal subspace is the subspace
\[
W(\Lambda)=U(\bar{\mathfrak{n}})\cdot v_\Lambda
\]
of a standard $A_n^{(1)}$-module $L(\Lambda)$, where $\mathfrak{n}$ is the nilradical of a Borel subalgebra of $sl_{n+1}$.
When $n=1$, the graded dimensions of $W(\Lambda_0)$ and $W(\Lambda_1)$ agree with the Rogers-Ramanujan functions.
The notion clearly extends to an arbitrary highest weight module for an affine Lie algebra.
The principal subspaces were studied in \cite{CLM2, CalLM3, G, AKS,CoLM,P1,FFJMM} and others.

Recently, Milas and Penn considered the lattice VOA $V_L$ and the vertex subalgebra
$W_L(B)=\langle e^{\beta_1},\ldots,e^{\beta_n} \rangle _{{\rm v.a.}}$, 
 called the principal subalgebra \cite{MP}.
This is a generalization of the principal subspaces of level one standard modules over simply-laced simple Lie algebras.
Combinatorial bases and the graded dimensions of the subalgebra $W_L(B)$ and some modules were given in \cite{MP}.

Let $F=F(B,\mathcal{N})$ be a free vertex superalgebra.
Then, $F$ is embedded in the lattice vertex superalgebra $V_Q$ associated with the lattice $Q=\langle B\rangle_{\mathbb{Z}}$ which is the free abelian group generated by $B$ \cite{R}.
Actually, the image is the principal subspace $W_Q(B)$ of $V_Q$ \cite{R}.

There is the generalization of the vertex superalgebras --- the {\it generalized vertex algebras} ({\it GVA}) \cite{DL,BK}.
One of the most important GVAs is a GVA associated with lattices which are not necessarily integral.

In this paper, we introduce the {\it free} GVAs and the {\it generalized principal subspaces}.
Our construction of free GVAs are similar to the construction of free vertex superalgebras in \cite{R}.
In order to generalize the construction of free vertex superalgebras, we use a generalization of {\it componentwise completion} of {\it graded algebras} \cite{MNT}.
The crucial point is that the formal power series $Y(a,z)$ ($a\in B$) is indeed a {\it field} (see Proposition  \ref{sec:propazfield}).

Then, we give combinatorial bases and the graded dimensions of free GVAs by generalizing the results of \cite{R,MP}.
Moreover, we show that the free GVA generated by a set $B$ is isomorphic to the generalized principal subspace associated with the free abelian group $Q_B$.

Finally, we show the duality of free GVAs (generalized principal subspaces).
This is related to the notion of the {\it dual lattices}.

In section 2, we generalize the notion of componentwise completions \cite{MNT}.
In section 3, we recall the generalized vertex algebras and 
give spanning sets of the weight graded GVAs.
In section 4, we construct the free generalized vertex algebras by using the componentwise completions in section 2.
In section 5, we introduce a notion of the generalized principal subspaces associated with lattices.
By using the vertex operators of lattice vertex algebras, we show that the spanning sets given in section 3 of the free GVAs are 
linearly-independent.
In section 6, we give a presentation of lattice GVAs in terms of generators and relations, by using the free GVAs.
In section 7, we compute the graded dimensions of the free GVAs and show the duality of free GVAs.
As an application, we describe the commutants of the Feigin-Stoyanovsky principal subspaces using the generalized principal subspaces (free GVAs).

\subsection*{Notations.}
We denote the non-negative integers by $\mathbb{Z}_+$ and the negative integers by $\mathbb{Z}_{<0}$.
All vector spaces are over a field $\mathbf{k}$ containing the complex numbers $\mathbb{C}$.
We endow the field $\mathbf{k}$ with the discrete topology.
For $z\in \mathbb{C}$, we denote the real part of $z$ by $\mathrm{Re}(z)$.

\section{Preliminaries on linear topologies}

We generalize the construction in \cite{MNT}.

\subsection{Componentwise topological algebras}

Let $Q$ be an abelian group equipped with a filtration
\[
\cdots \subset \mathrm{F}_{-1}Q \subset \mathrm{F}_0 Q\subset \mathrm{F}_1 Q \subset \mathrm{F}_2 Q\subset\cdots
\]
indexed by integers such that $\mathrm{F}_p Q +\mathrm{F}_q Q\subset \mathrm{F}_{p+q} Q$.
Suppose that the filtration is separated, that is, $\bigcap_{p\in\mathbb{Z}}\mathrm{F}_p Q=0$.
Let $A$ be an algebra and suppose given a grading
\[
A=\bigoplus_{\alpha\in Q} A(\alpha)
\]
indexed by the abelian group $Q$ such that $A(\alpha)\cdot A(\beta)\subset A(\alpha+\beta)$.
Define a filtration of $A$ by $\mathrm{F}_p A=\bigoplus_{\alpha\in \mathrm{F}_p Q} A(\alpha)$.
We simply call such an $A$ a {\it graded algebra}.

Since the filtration $\mathrm{F}$ is separated for  $\alpha\in Q$, we have the unique minimum 
$p\in \mathbb{Z}$ such that $\alpha\in \mathrm{F}_p Q$.
We denote such an integer by $p(\alpha)$.

Let $A=\bigoplus_{\alpha} A(\alpha)$ be a graded algebra and suppose given a linear topology on each $A(\alpha)$.
Let $n$ be a non-negative integer.
Consider the subspace
\[
A(\alpha) \cap (A\cdot \mathrm{F}_{-n-1}A)
\]
and let $\mathrm{I}_n(A(\alpha))$ be its closure in $A(\alpha)$.

\begin{dfn}
A {\it componentwise topological algebra} is a graded algebra $A$ endowed with 
a linear topology on each $A(\alpha)$ such that
\begin{enumerate}
\item the multiplication maps $A(\alpha)\times A(\beta)\rightarrow A(\alpha+\beta)$
are continuous;
\item for each $\alpha\in Q$, the sequence $\{\mathrm{I}_n(A(\alpha))\}$ forms a neighborhood basis of zero in $A(\alpha)$.
\end{enumerate}
A {\it componentwise complete algebra} is a componentwise topological algebra $A$
such that the topology on each $A(\alpha)$ is complete.
\end{dfn}

Let $A$ be a componentwise topological algebra.
Since the multiplication maps $A(\alpha)\times A(\beta)\rightarrow A(\alpha+\beta)$ are continuous,
we have
\[
A(\alpha)\cdot \mathrm{I}_n(A(\beta))\subset \mathrm{I}_n(A(\alpha+\beta))
\]
and
\[
\mathrm{I}_n(A(\alpha))\cdot A(\beta)\subset \mathrm{I}_{n-p(\beta)} (A(\alpha+\beta)).
\]
Therefore, for $a\in A(\alpha)$ and $b\in A(\beta)$, we have
\[
(a+\mathrm{I}_{n+p(\beta)}(A(\alpha)))\cdot (b+\mathrm{I}_n(A(\beta)))\subset a\cdot b + \mathrm{I}_n(A(\alpha+\beta)).
\]
We call the sum of the closures of the homogeneous subspaces of a graded subspace $U$ the
{\it componentwise closure} of $U$.

\subsection{Componentwise completions}
Let $A$ be a componentwise topological algebra.
Set
\[
\hat{A}=\bigoplus_{\alpha} \hat{A}(\alpha),
\]
where $\hat{A}(\alpha)$ is the completion 
$\varprojlim_n (A(\alpha)+\mathrm{I}_n(A(\alpha)))/\mathrm{I}_n(A(\alpha))$ of the space $A(\alpha)$.
We call $\hat{A}$ the {\it componentwise completion} of $A$.
Since the multiplication $A(\alpha)\times A(\beta)\rightarrow A(\alpha+\beta)$ are continuous,
they induce continuous bilinear maps $\hat{A}(\alpha)\times \hat{A}(\beta)\rightarrow \hat{A}(\alpha+\beta)$
which make $\hat{A}$ into an algebra endowed with a topology on each $\hat{A}(\alpha)$.
Explicitly, let $v$ and $w$ be elements of $\hat{A}(\alpha)$ and $\hat{A}(\beta)$.
Then, $v$ and $w$ have the forms  
$v=\prod_n (v_n+\mathrm{I}_n(A(\alpha)))$ and $w=\prod_n(w_n+\mathrm{I}_n(A(\beta)))$
with $v_n\in A(\alpha)$ and $w_n\in A(\beta)$ ($n\in \mathbb{Z}_+$).
The multiplication of $v$ and $w$ is given by
\[
v\cdot w= \prod_n (v_{n+p(\beta)}\cdot w_n +\mathrm{I}_n (A(\alpha+\beta))).
\]
Define the filtration $\hat{\mathrm{F}}$ of $\hat{A}$ by 
$\hat{\mathrm{F}}_p \hat{A}$ being the componentwise closure of $\mathrm{F}_p A$ for each $p\in \mathbb{Z}$.

\begin{prop}
The componentwise completion $\hat{A}$ is a componentwise complete algebra.
\end{prop}

Let $A=\bigoplus_{\alpha} A(\alpha)$ be a graded algebra.
Recall the filtration $\mathrm{F}_p A=\bigoplus_{\alpha\in Q_p} A(\alpha)$.
We endow the space $A(\alpha)$ with the linear topology defined by $\mathrm{I}_n(\alpha)=A(\alpha)\cap(A\cdot \mathrm{F}_{-n-1}A)$.
Then $A$ becomes a componentwise topological algebra.

We call this componentwise topology on $A$ the {\it standard componentwise topology} and
the completion $\hat{A}$ the {\it standard componentwise completion}.

\section{The generalized vertex algebras}

\subsection{Fields, locality and generalized vertex algebras}
The generalized vertex algebras were first introduced in \cite{DL}.
Following  \cite{BK}, we define a notion of generalized vertex algebras.

Let $U$ be a vector space.
We identify the subsets of $\mathbb{C}/\mathbb{Z}$ and the $\mathbb{Z}$-invariant subsets of $\mathbb{C}$.
Let $\Gamma$ be a subset of $\mathbb{C}/\mathbb{Z}$.
We denote by $U[[z,z^{\Gamma}]]$ the space of all formal infinite series $\sum_{n\in \Gamma}f(n) z^n$ with $f(n)\in U$.
We denote by $U[[z]]z^\Gamma$ the space of all finite sums of the form $\sum_i \psi_i(z)z^{d_i}$ with $d_i\in\Gamma$ and
$\psi_i(z)\in U[[z]]$.

Let $Q$ be an abelian group with a symmetric bilinear map
 $\Delta :Q\times Q\rightarrow \mathbb{C}/\mathbb{Z}$.
Let $V$ be a vector space, and suppose given a $Q$-grading
$V=\bigoplus_{\alpha\in Q}V^\alpha$ on $V$.

For $N\in \mathbb{C}$, we define the formal expansions
\begin{eqnarray*}
\iota_{z_1,z_2} (z_1-z_2)^N &:=& e^{-z_2 \partial_{z_1}} z_1^N \\
&=& \sum_{j\in\mathbb{Z}_+} \begin{pmatrix}N\\ j\end{pmatrix} z_1^{N-j} (-z_2)^j
 \in (\mathbb{C}[[z_1]]z_1^{N+\mathbb{Z}})[[z_2]],
\end{eqnarray*}
and
\begin{eqnarray*}
\iota_{z_2,z_1} (z_1-z_2)^N &:=& e^{\pi i N}e^{-z_1 \partial_{z_2}} z_2^N \\
&=& e^{\pi i N}\sum_{j\in\mathbb{Z}_+} \begin{pmatrix}N\\ j\end{pmatrix} z_2^{N-j} (-z_1)^j
 \in (\mathbb{C}[[z_2]]z_2^{N+\mathbb{Z}})[[z_1]].
\end{eqnarray*}

\begin{dfn}(\cite{BK})
\begin{enumerate}
\item A {\it field} of {\it charge} $\alpha\in Q$ on $V$ (with respect to $Q$, $\Delta$ and a grading $\bigoplus_{\alpha} V^\alpha$ on $V$)
is a formal series $a(z)\in (\mathrm{End} (V)) [[z,z^{\mathbb{C}}]]$
with the property that
\[
a(z)b \in V^{\alpha+\beta}[[z]]z^{-\Delta(\alpha,\beta)} \quad \quad \mathrm{for} \quad b\in V_\beta.
\]
We  denote by $\mathcal{F}^\alpha (V,Q,\Delta)$ the vector space of all fields of charge $\alpha$.
We denote $\mathcal{F}(V,Q,\Delta)=\bigoplus_{\alpha\in Q}\mathcal{F}^\alpha (V,Q,\Delta)$.

\item Let $p$ be a non-zero complex number. Two fields $a(z)$ and $b(z)$ of charges $\alpha$ and $\beta$
are called {\it mutually $p$-local} if there exists $N\in \Delta(\alpha,\beta)$ such that
\begin{equation}\label{eqn:local1}
\iota_{z,w}(z-w)^N a(z)b(w)=p\cdot \iota_{w,z}(z-w)^N b(w)a(z).
\end{equation}
We call such an $N$ a {\it ($p$-)locality bound} of $a(z)$ and $b(z)$.
\end{enumerate}
\end{dfn}

The following lemma states that the notion of $p$-locality bound is independent of a choice of $p$.
Let $a(z)$ and $b(z)$ be fields of charge $\alpha$ and $\beta$.
Let $p$ be non-zero complex numbers and $N$ a $p$-locality bound.

\begin{lem}
One of the following is hold:

(i) for any $q\in\mathbb{C}^\times$, the number $N$ is a $q$-locality bound;

or

(ii) for any $q\in\mathbb{C}^\times$ with $p\neq q$, the number $N$ is not a $q$-locality bound.
\end{lem}

\begin{proof}
If $\iota_{z,w}(z-w)^N a(z)b(w)=0$, then (i) holds.
Otherwise, (ii) holds.
\end{proof}

By the above lemma, we call a $p$-locality bound simply a locality bound.

\begin{dfn}
A {\it weak generalized vertex algebra} ({\it weak GVA}) consists of the following {\bf data}:
\begin{enumerate}
\item {\it (space of states)} a vector space $V$;
\item {\it (vacuum vector)} a non-zero vector $|0\rangle \in V$;
\item {\it (translation operator)} an endomorphism $T\in \mathrm{End} (V)$;
\item {\it (state-field correspondence)} a  linear map $Y:V\rightarrow \mathrm{End}(V)[[z,z^\mathbb{C}]]$,
$a\mapsto Y(a,z)=\sum_{n\in \mathbb{C}} a(n) z^{-n-1}$ ($a(n)\in \mathrm{End}(V)$, $n\in \mathbb{C}$) from $V$ to the space of formal series,
\end{enumerate}
subject to the following {\bf axioms}:
\begin{enumerate}
\item {\it (vacuum axiom)} $T|0\rangle =0$, and $Y(a,z)|0\rangle|_{z=0}=a$ ($a\in V$);
\item {\it (translation covariance)} $[T,Y(a,z)]=\partial_z Y(a,z)$ ($a\in V$).
\end{enumerate}
\end{dfn}

We denote the weak GVAs by $(V,Y,|0\rangle,T)$ or $V$.

Let $V_1,V_2$ be weak GVAs.

\begin{dfn}
A {\it homomorphism} of weak GVAs is a linear map $f:V_1\rightarrow V_2$
which satisfies
\begin{enumerate}
\item $f(Y_1(a,z)b)=Y_2(f(a),z)f(b)$ ($a,b\in V_1$);
\item $f(|0\rangle_1)=|0\rangle_2$;
\item $f\circ T_1=T_2\circ f$.
\end{enumerate}
\end{dfn}

\begin{dfn}
\begin{enumerate}
\item A {\it weak generalized vertex subalgebra} ({\it subweak GVA}) of a weak GVA $V$ is
a $T$-invariant subspace $W$ of $V$ such that $|0\rangle\in W$ and that  $Y(w_1,z)w_2\in W[[z,z^\mathbb{C}]]$ for $w_1,w_2\in W$
 (with the induced weak GVA structure).
\item A {\it weak GVA (left) ideal} of a weak GVA $V$ is a $T$-invariant subspace $I$ of $V$ such that 
$Y(v,z)u\in I[[z,z^\mathbb{C}]]$ for any $v\in V$, $u\in I$.
\end{enumerate}
\end{dfn}

Let $V,W$ be weak GVAs and $f:V\rightarrow W$ a homomorphism of weak GVAs.
Then, the image of $f$ is a subweak GVA of $W$, and the kernel of $f$ is a weak GVA ideal of $V$.
Let $I,J$ be weak GVA ideals of $V$.
Then, $I\cap J$ and $I\oplus J$ are also  weak GVA ideals of $V$.
Let $U,U'$ be subweak GVA of $V$.
Then, $U\cap U'$ is also a subweak GVA of $V$.
Let $B$ be a subset of $V$.
The weak GVA ideal {\it generated by $B$} is the unique minimal weak GVA ideal of $V$ containing $B$.
The subweak GVA {\it generated by $B$} is the unique minimal subweak GVA of $V$ containing $B$.

Now we define a notion of the generalized vertex algebras.

\begin{dfn}
A {\it generalized vertex algebra} ({\it GVA}) is a weak GVA $(V,Y,|0\rangle,T)$ such that
\begin{itemize}
\item {\it (field and locality axiom)} there exist 
\begin{enumerate}
\item an abelian group $Q$;
\item a symmetric bilinear map $\Delta:Q\times Q\rightarrow \mathbb{C}/\mathbb{Z}$;
\item a $Q$-grading $V=\bigoplus_{\alpha\in Q}V^\alpha$ on $V$;
\item a bimultiplicative function $\eta:Q\times Q\rightarrow \mathbb{C}^\times$,
\end{enumerate}
 such that
\begin{enumerate}
\item the vacuum vector $|0\rangle$ belongs to $V^0$;
\item the translation operator $T$ is $Q$-grading preserving;
\item the state-field correspondence $Y$ becomes a $Q$-grading preserving linear map into $\mathcal{F}(V,Q,\Delta)=\bigoplus_\alpha \mathcal{F}^\alpha (V,Q,\Delta)$;

(that is, for any $\alpha\in Q$ and $a\in V^\alpha$, the formal series $Y(a,z)$ is a field of charge $\alpha$ with respect to
 $Q$, $\Delta$ and the $Q$-grading on $V$);
\item the complex number $\eta(\alpha,\beta)\eta(\beta,\alpha)$ is equal to $e^{-2\pi i \Delta(\alpha,\beta)}$ ($\alpha,\beta\in Q$);
\item the fields $Y(a,z)$ and $Y(b,w)$ are mutually $\eta(\alpha,\beta)$-local ($\alpha,\beta\in Q$, $a\in V^\alpha$, $b\in V^\beta$).
\end{enumerate}
\end{itemize}
\end{dfn}

We denote the generalized vertex algebra by $(V,Y,|0\rangle,T)$ or $V$.
We call a quadruple $(Q,\Delta,(Q \mbox{-grading}\, V=\bigoplus_\alpha V^\alpha), \eta)$ in the field and locality axiom a {\it charge factor on $V$}.
We denote it by $(Q,\Delta,\eta)$ or $Q$.

\begin{rem}
Let $(V,Y,|0\rangle,T)$ be a GVA and $(Q,\Delta,\eta)$ a charge factor on $V$.
Then, for fixed $Q$ and $\Delta$, the quadruple $(V=\bigoplus_{\alpha\in Q}V^\alpha,Y,|0\rangle,T)$ is a generalized vertex algebra in \cite{BK}.
\end{rem}

Let $V$ be a GVA.
Fix a charge factor $(Q,\Delta,\eta)$ on $V$.

We call the grading $V=\bigoplus_{\alpha\in Q} V^\alpha$ a {\it charge-grading} and denote $\gamma(a)=\alpha$ for any non-zero $a\in V^\alpha$.
We call the bimultiplicative function $\eta:Q\times Q\rightarrow \mathbb{C}^\times$
a {\it locality factor}.

Let $B$ be a subset of $V$ consisting of $Q$-homogeneous elements and $\mathcal{N}:B\times B\rightarrow \mathbb{C}$ be a symmetric function.
We call $\mathcal{N}$ a {\it locality bound} on $B$ if for any $a,b\in B$, the number $\mathcal{N}(a,b)$ is a locality bound of $Y(a,z)$ and $Y(b,z)$.

Let $a$ and $b$ be homogeneous elements of $V$ with the charges $\alpha$ and $\beta$.
Note that the two fields $Y(a,z)=\sum a(n)z^{-n-1}$ and $Y(b,w)=\sum b(n)w^{-n-1}$ are $\eta(\alpha,\beta)$-local if and only if
there exists $N\in\Delta(\alpha,\beta)$ such that  for ant $s,t\in \mathbb{C}$,
\begin{eqnarray}\label{eqn:local2}
&&\sum_{j\in \mathbb{Z}_+} (-1)^j \begin{pmatrix} N\\ j\end{pmatrix} \biggl(
a(s+N-j)b(t+j) \\ \nonumber
&& \quad\quad\quad -\eta(\alpha,\beta) \cdot e^{\pi i N} b(t+N-j) a(s+j)\biggr) =0.
\end{eqnarray}
We denote the LHS of (\ref{eqn:local2}) by $l(a,b,s,t,N,\eta)$.

Let $(V_1,Y_1,|0\rangle_1,T_1)$ and $(V_2,Y_2,|0\rangle_2,T_2)$
be generalized vertex algebras.

\begin{dfn}
A {\it homomorphism} of generalized vertex algebras is a homomorphism of weak GVAs $f:V_1\rightarrow V_2$
such that
\begin{itemize}
\item ({\it charge compatibility}) there exist 
\begin{enumerate}
\item charge factors $(Q_1,\Delta_1,\eta_1)$ on $V_1$ and $(Q_2,\Delta_2,\eta_2)$ on $V_2$;
\item a group-homomorphism $\phi:Q_1\rightarrow Q_2$
\end{enumerate}
which satisfy the following conditions ($\alpha,\beta\in Q_1$):
\begin{enumerate}
\item $f((V_1)^\alpha)\subset (V_2)^{\phi(\alpha)}$;
\item $\Delta_2(\phi(\alpha),\phi(\beta))=\Delta_1(\alpha,\beta)$;
\item $\eta_2(\phi(\alpha),\phi(\beta))=\eta_1(\alpha,\beta)$.
\end{enumerate}
\end{itemize}
\end{dfn}

Note that when charge factors $(Q_1,\Delta_1,\eta_1)$ and $(Q_2,\Delta_2,\eta_2)$ and a group-homomorphism $\phi:Q_1\rightarrow Q_2$ satisfy the charge compatibility axiom, then 
 $Q_2,\Delta_2,\eta_2$ and the induced grading $V_1=\bigoplus_{\beta\in Q_2}(V_1)^\beta$, where $(V_1)^\beta=\bigoplus_{\alpha\in Q_1,\phi(\alpha)=\beta}(V_1)^\alpha$,
forms a charge factor on $V_1$.
The charge factors $(Q_2,\Delta_2,\eta_2)$ on $V_1$ and $V_2$ and the identity map $Q_2\rightarrow Q_2$ satisfy
again the charge compatibility axiom.

\begin{dfn}
\begin{enumerate}
\item A {\it generalized vertex subalgebra} ({\it subGVA}) of a generalized vertex algebra $V$ is
a subweak GVA $W$ of $V$ such that there exists a charge factor $Q$ on $V$ such that $W$ is a $Q$-graded subspace of $V$
 (with the induced GVA structure).
\item An {\it ideal} ({\it GVA-ideal}) of a generalized vertex algebra $V$ is a weak GVA ideal $I$ of $V$ such that 
 there exists a charge factor $Q$ on $V$ such that $I$ is a $Q$-graded subspace of $V$.
\end{enumerate}
\end{dfn}

Let $V$ be a GVA and $W$ a subGVA.
By the definition, the natural injection $\iota: W\hookrightarrow V$ is a homomorphism of GVAs.

Let $I$ be an ideal of $V$.
Consider the quotient space $V/I$ and
set $Y'(a+I,z)(b+I)=Y(a,z)b+I$ for $a,b\in I$.
Since the GVA-ideals are ``two-sided" ideals, it defines a well-defined map from $V/I\times V/I$ to $\mathrm{End}(V/I)[[z,z^\mathbb{C}]]$.
Since $I$ is $T$-invariant, the operator $T$ induces on $V/I$ the operator $T'$.

\begin{prop}
The induced weak GVA $(V/I,Y',|0\rangle+I,T')$ is a generalized vertex algebra.
Moreover, the projection $\pi:V\rightarrow V/I$ is a homomorphism of GVAs.
\end{prop}

\begin{proof}
Let $(Q,\Delta,\eta)$ be a charge factor on $V$ such that $I$ is $Q$-invariant.
Then, we have the induced grading $V/I=\bigoplus_{\alpha\in Q} (V/I)^\alpha$.
The quadruple $(Q,\Delta, (\mbox{the grading on}\, V/I), \eta)$ satisfies the field and locality axiom on $V/I$.
Therefore, $V/I$ is a GVA.
The projection $\pi:V\rightarrow V/I$ is a homomorphism with the group-homomorphism $\mathrm{id}:Q\rightarrow Q$.
\end{proof}

Let $V_1,V_2$ be GVAs and $f:V_1\rightarrow V_2$ a homomorphism of GVAs.

\begin{prop}
The kernel of $f$ is an ideal of $V_1$. Moreover, the image of $f$ is a subGVA of $V_2$.
\end{prop}

\begin{proof}
Let $(Q_1,\Delta_1,\eta_1)$ and $(Q_2,\Delta_2,\eta_2)$ be charge factors on $V_1$ and $V_2$ and $\phi:Q_1\rightarrow Q_2$ be a group-homomorphism satisfying 
the charge compatibility axiom on the homomorphism $f$.
The image of $f$ is a subVOA of $V_2$ with the charge factor $(Q_2,\Delta_2,\eta_2)$ (with the induced grading on $W$).
Now, we show that $ker(f)$ is an ideal of $V_1$.
We have $Y(\ker(f),z)V\subset \ker(f)[[z,z^\mathbb{C}]]$ and $T(\ker(f))\subset \ker(f)$.
Let $v$ be an element of $\ker(f)$.
Consider the induced grading  $V_1=\bigoplus_{\beta\in Q_2}(V_1)^\beta$, where $(V_1)^\beta=\bigoplus_{\alpha\in Q_1,\phi(\alpha)=\beta}(V_1)^\alpha$.
Then the quadruple $(Q_2,\Delta_2, ($the induced grading on $V_1),\eta_2)$ is a charge factor on $V_1$.
The charge factors together with the identity map $Q_2\rightarrow Q_2$ satisfy again the charge compatibility axiom.
The vector $v$ has the form $v=\sum_{\beta\in Q_2} v^\beta$ 
with $v^\beta\in (V_1)^\beta$.
Then, by the charge compatibiliy axiom, we have $f(v^\beta)=0$ for each $\beta\in Q_2$.
Hence, $\ker(f)$ is a $Q_2$-graded subspace, which completes the proof.
\end{proof}

Let $V$ be a GVA and $(Q,\Delta,\eta)$ a charge factor on $V$.
Let $B$ be a subset of $V$.
Suppose that each element of $B$ is a $Q$-homogeneous vector.

\begin{prop}
The weak GVA ideal generated by $B$ is a GVA ideal of $V$.
The subweak GVA generated by $B$ is a subGVA of $V$.
\end{prop}

\begin{proof}
We show the former statement.
Let $I$ denote the weak GVA ideal generated by $B$.
By the associativity of the GVAs (\cite{BK}), the space $I$ is spanned by the monomials of the form
$
a_1(n_1)\ldots a_m(n_m)b
$
with a non-negative integer $m\geq 0$, $Q$-homogeneous elements $a_1,\ldots,a_m \in V$, 
complex numbers $n_1,\ldots,n_m \in \mathbb{C}$ and a vector $b\in B$.
Hence, $I$ is $Q$-graded, which completes the proof.
\end{proof}

Let $Q$ be an abelian group, $V$ a $Q$-graded vector space, $|0\rangle$ an element of $V$ and $T$ a linear map on $V$.
Let $\Delta:Q\times Q\rightarrow \mathbb{C}/\mathbb{Z}$ be a symmetric bilinear map and
 $\eta:Q\times Q\rightarrow \mathbb{C}^\times$  a bimultiplicative function with $\eta(\alpha,\beta)\eta(\beta,\alpha)=e^{-2\pi i \Delta(\alpha,\beta)}$.
Let $\{\phi_i(z)\}_{i\in I}$ be a system of fields on $V$.
Denote the charge of $\phi_i(z)$ by $\alpha_i$ .
Assume $\phi_i(z)$ and $\phi_j(z)$ are mutually $\eta(\alpha_i,\alpha_j)$-local for every $i,j\in I$.
Assume $\phi_i(z)$ is translation covariant with respect to $T$, that is, $[T, \phi_i(z)]=\partial_z \phi_i(z)$, for every $i\in I$.
Assume the coefficients of all formal series $\phi_{i_1}(z_1)\cdots \phi_{i_n}(z_n)|0\rangle$ ($n\in \mathbb{Z}_+$) span the whole vector space $V$.

\begin{thm}\label{sec:thmbk}(Reconstruction theorem \cite{BK})
The fields $\{\phi_i(z)\}$ generates on $V$ a unique structure of a generalized vertex algebra with a charge factor $(Q,\Delta,\eta)$.
\end{thm}

Sometimes, we use another grading.

\begin{dfn}
The $\mathbb{C}$-grading $V=\bigoplus_{n\in\mathbb{C}}V_n$ is called a {\it weight grading} if
\begin{itemize}
\item there exists a charge factor $(Q,\Delta,\eta)$ such that the grading is compatible with the $Q$-grading $V=\bigoplus_{\alpha\in Q}V^\alpha$ (we call such a charge factor {\it compatible} with the weight-grading);
\item $a(k) V_n\subset V_{m+n-k-1} (a\in V_m, k\in \mathbb{C})$.
\end{itemize}
\end{dfn}

\subsection{The current algebras}

Let $(V,Y,|0\rangle,T)$ be a generalized vertex algebra equipped with a weight-grading $V=\bigoplus_{d\in\mathbb{C}} V(d)$.
Fix a charge factor $(Q,\Delta,\eta)$ compatible with the weight-grading.
Assume that the weight-grading is bounded below, that is, there exists $N\in\mathbb{R}$ such that
$V=\bigoplus_{d\in\mathbb{C},\mathrm{Re}(d)\geq N} V(d)$.

Let us consider the space
\[
L(V)=V[t,t^{-1}]=V\otimes_{\mathbf{k}}\mathbf{k}[t,t^{\mathbb{C}}].
\]
Let $a$ be a weight-homogeneous vector of $V$ and $m$ a complex number. 
We denote $a(m)=a\otimes t^m$.
We set $\mathrm{wt}(a(m))=\mathrm{wt}(a) -m-1$.
Then we have the grading $L(V)=\bigoplus_{d\in\mathbb{C}} (L(V))(d)$ on L(V).
Consider the quotient space
\[
W=L(V)/\partial L(V),
\]
where $\partial:L(V)\rightarrow L(V)$ is defined by
\[
\partial (u\otimes t^n)=T u\otimes t^n+n u\otimes t^{n-1}.
\]
Then $W$ admits the grading $W=\bigoplus_{d} W(d)$.

Let $A$ denote the symmetric algebra of $W$.
The grading on $W$ induces the grading $A=\bigoplus_{d\in\mathbb{C}} A(d)$ on $A$.
Define a filtration by $\mathrm{F}_p A=\bigoplus_{d\in\mathbb{C}, \mathrm{Re}(d)\leq p} A(d)$ ($p\in \mathbb{Z}$).
Since it is separated,
$A$ becomes a filtered graded algebra.
Consider the standard filtered componentwise topology on $A$ and let $\hat{A}$ denote the filtered componentwise completion.

Let $a,b$ be weight- and charge-homogeneous element of $V$ with the charges $\alpha$ and $\beta$.
Consider the following relations in $\hat{A}$ for $m,k\in \mathbb{C}$ and $n\in \Delta(\alpha,\beta)$.
\begin{eqnarray*}
\mathrm{B}_{m,n,k}(a,b)&=&\sum_{j\in\mathbb{Z}_+} \begin{pmatrix}m\\ j \end{pmatrix} (a(n+j) b)(m+k-j) 
\\ && - \sum_{j\in\mathbb{Z}_+} (-1)^j \begin{pmatrix}n\\ j \end{pmatrix} \bigl( a(m+n-j) b(k+j)
\\ &&-\eta(\alpha,\beta)
e^{\pi i n} b(n+k-j) a(m+j)\bigr).
\end{eqnarray*}

The first sum in the RHS is actually a finite sum whereas the second and the last are infinite sums which 
converge in the linear topology of $\hat{A}(\mathrm{wt}(a)+\mathrm{wt}(b)-n-m-k-2)$.
The relations $\mathrm{B}_{m,n,k}(a,b)=0$ are nothing else but the Borcherds identities.

Let $\mathbb{B}$ be the two-sided ideal of $\hat{A}$ generated by the elements $\mathrm{B}_{m,n,k}$, and 
let $\hat{\mathbb{B}}$ be the componentwise closure of $\mathbb{B}$.
Then $\mathbb{B}$ is also an ideal of $\hat{A}$.

We now define the {\it current algebra} $U$ associated with $V$ to be the quotient algebra of $\hat{A}$ by the ideal $\hat{\mathbb{B}}$:
\[
U=\hat{A}/\hat{\mathbb{B}}.
\]
Then U is a bi-graded algebra, since $\hat{\mathbb{B}}$ is a graded ideal.

\subsection{Some lemmas}

Now, we show some lemmas for the following section.

Let $V$ be a GVA with a charge factor $(Q,\Delta,\eta)$.

Let $a,b,c$ be $Q$-homogeneous elements of $V$.
Let $\alpha$ and $\beta$ denote the charge of $a$ and $b$.
Let $N$ be a locality bound of $Y(a,z)$ and $Y(b,z)$.
Let $s,t$ be complex numbers.

\begin{lem}\label{sec:lemlocal3}
\begin{eqnarray}\label{eqn:local3}
a(s) b(t) c=\eta(\alpha,\beta)e^{\pi i N}\sum_{j\in \mathbb{Z}_+} (-1)^j \begin{pmatrix} N\\ j\end{pmatrix} b(t+N-j) a(s-N+j)c
\nonumber \\
\quad\quad-\sum_{j\in \mathbb{Z}_+} (-1)^{j+1} \begin{pmatrix} N\\ j+1\end{pmatrix}
a(s-1-j)b(t+1+j)c.
\end{eqnarray}
\end{lem}

\begin{proof}
By equality (\ref{eqn:local2}), we have the lemma.
\end{proof}

\begin{lem}\label{sec:lemlocal4}
\begin{eqnarray*}
a(s) b(t) c=\sum_{j=-M}^M r_j \cdot b(t+N-j)a(s-N+j)c,
\end{eqnarray*}
where $M$ is a non-negative integer and $r_{-M},\ldots,r_M$ are complex numbers.
\end{lem}

\begin{proof}
Since $Y(b,z)$ is a field, by Lemma \ref{sec:lemlocal3}, we have the lemma.
\end{proof}

\subsection{A spanning set of the generalized vertex algebras}
\label{sec:secspan}
We construct a spanning set of certain generalized vertex algebras.
Let $(V,Y,|0\rangle, T)$ be a generalized vertex algebra and $(Q,\Delta,\eta)$ be a charge factor on $V$.
Let $B$ be a totally ordered set with a map $i:B\rightarrow V$.
Suppose that each element of $i(B)$ is a non-zero $Q$-homogeneous vector.
Suppose that the subGVA generated by the subset $i(B)\subset V$ agrees with $V$.

Let $\mathcal{N}:B\times B\rightarrow \mathbb{C}$ be a symmetric map.
For any $a,b\in B$, suppose that the number $\mathcal{N}(a,b)$ is a locality bound of the fields $Y(i(a),z)$ and $Y(i(b),z)$.
We call such a map $\mathcal{N}$ a {\it locality bound} on $(B,i)$.
We denote $a=i(a)$ ($a\in B$).

For a non-empty finite ordered set $X=\{n_1<\cdots<n_l\}$ and operators $f_n$ ($n\in X$), we denote the operator $f_{n_l}\circ \cdots \circ f_{n_1}$ by 
$\prod_{n\in X} f_n$.
When $X$ is empty, $\prod_{n\in X} f_n$ denotes the identity operator.

Let $k$ be a non-negative integer.
Let $a_1,\ldots,a_k$ be elements of the set $B$ and $m_1,\ldots,m_k$ complex numbers.
Define $\Psi(a_1,m_1;\cdots;a_k,m_k)$  by recursively setting
\begin{eqnarray*}
&&\Psi(a_1,m_1;a_2,m_2;\cdots;a_k,m_k)\\
&& \quad\quad := a_k( m_k+{\textstyle \sum_{j=1}^{k-1}\mathcal{N}(a_k,a_j)}) \Psi(a_1,m_1;\cdots;a_{k-1},m_{k-1})
\end{eqnarray*}
and
\[
\Psi(\ ):=|0\rangle \quad\quad (k=0).
\]
For a non-negative integer $r\leq k$ and monomial $u=\Psi(a_1,m_1;\cdots;a_r,m_r)$, we put
\[
\Psi(u;a_{r+1},m_{r+1};\cdots;a_k,m_k):=\Psi(a_1,m_1;\cdots;a_k,m_k).
\]

Let $\leq$ denote the total order of the totally ordered set $B$.

Let $k$ be a non-negative integer.
Let $a_1,\ldots,a_k$ be elements of $B$ with $a_1\leq \cdots \leq a_k$.
Let $M$ be an integer.
Consider the subspace $
V(a_1,\ldots,a_k;M)\subset V
$
spanned by the elements
\begin{equation*}
\Psi(a_1,m_1;\cdots;a_k,m_k)=\prod_{i=1}^{k}\left(a_i( m_i+{\textstyle \sum_{j=1}^{i-1}\mathcal{N}(a_i,a_j)})
 \right) |0\rangle
\end{equation*}
with $m_i\in \mathbb{Z}$ ($i=1,\ldots,k$) such that
$m_1+\cdots+m_k=M$.
Note that when $k=0$, we have $V(\ ;M)=0$ if $M\neq 0$ and $V(\ ;0)=\mathbb{C}|0\rangle$.

\begin{lem}\label{sec:lemgen}
\[
V=\sum_{k\geq 0, a_1\leq \cdots\leq a_k\in B, M\in \mathbb{Z}} V(a_1,\ldots,a_k;M).
\]
\end{lem}

\begin{proof}
Since $B$ generates $V$, by the associativity of the generalized vertex algebras and Lemma \ref{sec:lemlocal4}, 
the space $V$ is spanned by the monomials of the form
$
v=\Psi(a_1,m_1;\cdots;a_k,m_k)
$
with $k\geq 0$ and $m_i\in \mathbb{C}$, $a_1,\ldots,a_k\in B$ with $a_1\leq \cdots \leq a_k$.
When $m_j \not \in \mathbb{Z}$ for some $j$, we have $v=0$, since $\Delta(a_i,a_j)=\mathcal{N}(a_i,a_j)+\mathbb{Z}$.
Thus, we have the lemma.
\end{proof}

Consider the subset $
\mathcal{C}(a_1,\ldots,a_k;M)
$
consisting of the elements
\[
\Psi(a_1,m_1;\cdots;a_k,m_k)=\prod_{i=1}^{k}\left(a_i( m_i+{\textstyle \sum_{j=1}^{i-1}\mathcal{N}(a_i,a_j)})
 \right) |0\rangle
\]
with negative integers $m_1,\ldots,m_k \in \mathbb{Z}_{<0}$ such that $m_1+\cdots+m_k=M$ and
 if $i<j$ and $a_i=a_j$, then $m_i\geq m_j$.
Note that if $M >-k$, then $\mathcal{C}(a_1,\ldots,a_k;M)$ is empty.
When $k=0$, we have $\mathcal{C}(\ ;M)=\emptyset$ if $M\neq 0$ and $\mathcal{C}(\ ;0)=\{|0\rangle\}$.

Set $\mathcal{C}(V,B,\mathcal{N})=\bigcup_{k\geq 0,a_1\leq \cdots\leq a_k\in B,M\leq -k} \mathcal{C}(a_1,\ldots,a_k;M)$.

\begin{thm}
Let $k$ be a non-negative integer, $M$ an integer and $a_1,\ldots,a_k$ elements of $B$ with $a_1\leq \cdots \leq a_k$.
\begin{enumerate}
\item The set $\mathcal{C}(a_1,\ldots,a_k;M)$ spans the vector space $V(a_1,\ldots,a_k;M)$.
\item If $M>-k$, then $V(a_1,\ldots,a_k;M)=0$.
\end{enumerate}
\end{thm}

\begin{cor}\label{sec:span}
The set $\mathcal{C}(V,B,\mathcal{N})$ spans the vector space $V$.
\end{cor}

The corollary follows from the theorem and Lemma \ref{sec:lemgen}.
In the rest of the section we prove the theorem.

\begin{proof}
Put $\mathcal{C}=\mathcal{C}(a_1,\ldots,a_k;M)$.
If $M>-k$ and $\mathcal{C}$ spans $V(a_1,\ldots,a_k;M)$, then $V(a_1,\ldots,a_k;M)=0$,
 since $\mathcal{C}$ is empty if $M>-k$.

Let $k$ be a non-negative integer, $M$ an integer and $a_1\leq \cdots \leq a_k$ elements of $B$.
Let $v$ be a monomial $\Psi(a_1,m_1;\cdots;a_k,m_k)$ with $m_1,\ldots,m_k\in \mathbb{C}$ with $m_1+\cdots+m_k=M$.
Put $\mathcal{C}=\mathcal{C}(a_1,\ldots,a_k;M)$.
It suffices to show $v\in \mathbb{C}$-span$(\mathcal{C})$ to prove (1) and (2).
We show it by induction on $k$.

Assume $k=1$.
If $M\leq -1$, then $v\in \mathcal{C}$.
If $M> -1$, then $v=0\in \mathbb{C}$-span$(\mathcal{C})$, since $a_1(j) |0\rangle=0$ for $j>-1$.

Assume $k\geq 2$.
The proof falls into two parts:

\begin{itemize}
\item $a_k=a_{k-1}$;
\item $a_k>a_{k-1}$.
\end{itemize}

\underline{{\bf The case $a_k=a_{k-1}$.}}
First, we prove the case $a_{k-1}=a_k$.
Put $v^0=v$.
By the induction hypothesis, if $M-m_k>-k+1$, then we have $v^0=0\in \mathbb{C}$-span$(\mathcal{C})$, since the vector
$\Psi(a_1,m_1;\cdots;a_{k-1},m_{k-1})$ belongs to $V(a_1,\ldots,a_{k-1};M-m_k)$.

Assume $M-m_k\leq -k+1$.
By the induction hypothesis,  $v^0$ is the linear sum of the monomials of the form
$
x^1=\Psi(u;a_k,m_k)=a_k( m_k+\sum_{j=1}^{k-1}\mathcal{N}(a_k,a_j)) u
$
with $u\in \mathcal{C}(a_1,\ldots,a_{k-1};M-m_k)$.
Then, the monomial $u$ has the form 
$
u=\Psi(a_1,n_1;\cdots;a_{k-1},n_{k-1})
$
with negative integers $n_1,\ldots, n_{k-1}\in \mathbb{Z}_{<0}$ such that $M-m_k=n_1+\cdots+n_{k-1}$ and if 
$i<j$ and $a_i=a_j$, then $n_i\geq n_j$.
When $m_k\leq n_{k-1}$, we have $x^1\in \mathcal{C}$.
Assume $m_k >n_{k-1}$.
Apply formula (\ref{eqn:local3}) to the monomial $x^1$ with $N=\mathcal{N}(a_k,a_{k-1})=\mathcal{N}(a_k,a_k)$.
Then,
$x^1$ becomes the sum of the monomials of the form
$
v^1=\Psi(a_1,m'_1;\cdots;a_k,m'_k)
$
with integers $m'_i\in \mathbb{Z}$ such that $m'_k<m_k$ and $M=m'_1+\cdots+m'_k$.
Since $m'_k<m_k$, the integer $M-m'_k$ is greater than $M-m_k$.
If $M-m'_k>-k+1$, then $v^1=0$.
Otherwise, again apply the procedure to $v^1$ and we obtain either $x^2\in\mathcal{C}$ or $v^2$ with $m''_i$ such that $m''_k<m'_k$.
Hence, by repeating the procedure, eventually we obtain 
$v \in \mathbb{C}$-span$(\mathcal{C})$.

\underline{{\bf The case $a_k>a_{k-1}$.}}
Now, we prove the case $a_k>a_{k-1}$.
We will apply formula (\ref{eqn:local3}) twice.
Put $m_{i,0}=m_i$ for $i=1,\ldots,k$.
When $m_{k,0}\leq -1$, by the induction hypothesis and $a_k>a_{k-1}$, we have $v\in\mathbb{C}$-span$(\mathcal{C})$.
Suppose $m_{k,0}> -1$.
Set $v^0=v$.
Apply formula (\ref{eqn:local3}) to the monomial $v$ with $N=\mathcal{N}(a_k,a_{k-1})$.
Then, $v^0$ becomes the sum of the monomials of the forms
\[
v^1=\Psi(a_1,m_{1,1};\cdots;a_k,m_{k,1})
\]
and
\begin{eqnarray*}
y&=&\Psi(a_1,m_{1,0};\cdots;a_{k-2},m_{k-2,0};a_k,q;a_{k-1},p)
\end{eqnarray*}
with integers $m_{1,1},\ldots,m_{k,1},p,q\in \mathbb{Z}$ with $m_{k,1}<m_{k,0}$, $M=m_{1,1}+\cdots+m_
{k,1}$ and $p+q=m_{k,0}+m_{k-1,0}$.

Now, we show $y\in\mathbb{C}$-span$(\mathcal{C})$.
Put $y^0=y$.
By the induction hypothesis, if $M-p>-k+1$, then $y^0=0$.
Assume $M-p\leq -k+1$.
By the induction hypothesis, $y^0$ is the sum of the monomials of the form
\[
z=\Psi( z';a_k,r;a_{k-1},p)
\]
with a negative integer $r$ and monomial $z'\in \mathcal{C}(a_1,\ldots,a_{k-2},M-p-r)$.
Then again apply formula (\ref{eqn:local3}) to $z$.
Then $z$ becomes the sum of the monomials of the forms
\[
g=\Psi(z';a_{k-1},p+l;a_k,r-l)
\]
and
\[
y^1=\Psi(z';a_k,r+l+1;a_{k-1},p-l-1)
\]
with non-negative integers $l\geq 0$.
By the induction hypothesis, $\Psi(z';a_{k-1},p+l)$
belongs to $\mathbb{C}$-span$(\mathcal{C}(a_1,\ldots,a_{k-1},M-r+l))$.
Since $r-l\leq -1$ and $a_k>a_{k-1}$, we have $g\in\mathbb{C}$-span$(\mathcal{C})$.
Since $l\geq 0$, the number $M-(p-l-1)$ is greater than $M-p$.
If $M-(p-l-1)>-k+1$, then $y^1=0$ by the induction hypothesis.
Otherwise, again apply the procedure to $y^1$.
By repeating the procedure, we eventually obtain $y\in\mathbb{C}$-span$(\mathcal{C})$.

Now, let us show $v^0\in \mathbb{C}$-span$(\mathcal{C})$.
Since $m_{k,1}<m_{k,0}$, the number $M-m_{k,1}$ is greater than $M-m_{k,0}$.
If $M-m_{k,1}>-k+1$, by the induction hypothesis, we have $v^1=0$.
Otherwise, again apply the procedure to $v^1$, then we obtain monomials $v^2$ with $m_{k,2}<m_{k,1}$.
By repeating the procedure, eventually we have $v=v^0\in \mathbb{C}$-span$(\mathcal{C})$.
Thus, we have the proposition.
\end{proof}

\section{The free generalized vertex algebras}

\subsection{The free generalized vertex algebras}
\label{sec:freegva}

Let $B$ be a set.
Let $Q_B=\bigoplus_{a\in B}\mathbb{Z}e(a)$ be a free abelian group with a $\mathbb{Z}$-basis $\{e(a)\}_{a\in B}$ indexed by $B$.
We identify $B$ and $\{e(a)\}_{a\in B}$.
Let $\mathcal{N}:B\times B\rightarrow \mathbb{C}$ be a symmetric function.
We bilinearly extend $\mathcal{N}$ to the symmetric bilinear form $\mathcal{N}:Q_B\times Q_B\rightarrow \mathbb{C}$.
Define $\Delta:Q_B\times Q_B\rightarrow \mathbb{C}/\mathbb{Z}$ to be $\Delta(\alpha,\beta)=\mathcal{N}(\alpha,\beta)+\mathbb{Z}$ ($\alpha,\beta\in Q_B$).
Let  $\eta:Q_B\times Q_B\rightarrow \mathbb{C}^\times$ be a bimultiplicative function
with $\eta(\alpha,\alpha)=e^{-\pi i \mathcal{N}(\alpha,\alpha)}$.
Note that we have $\eta(\alpha,\beta)\eta(\beta,\alpha)=e^{-2\pi i \Delta(\alpha,\beta)}$ for any $\alpha,\beta\in Q_B$.

Now we construct a free generalized vertex algebra $F=F(B,\mathcal{N},\eta)$.

Set $X=\{a(n)|a\in B, n\in\mathbb{C}\}(=B\times \mathbb{C})$.
Let $A=\mathbf{k}[X]$ denote the free associative algebra generated by $X$.
We denote the unit by $|0\rangle$.
Let $m$ denote a positive integer, $a_1,\ldots,a_m$ elements of $B$ and $n_1,\ldots,n_m$ elements of $\mathbb{C}$.
Set $T(|0\rangle)=0$ and $T(a_1(n_1) \cdots a_m(n_m))=\sum_{i=1}^m (-n_i)\cdot a_1(n_1)\cdots a_i(n_i-1)\cdots a_m(n_m)$ , and extend $T$ linearly on $A$.
The space $A$ is $Q_B$-graded by 
$\gamma(a_1(n_1)\cdots a_m(n_m))=e(a_1)+\cdots+e(a_m)$.
Actually, $A$ is $(Q_B)_+$-graded, where $(Q_B)_+=\bigoplus_{a\in B} \mathbb{Z}_+ e(a)$.
We define a $\mathbb{C}$-grading by 
$\delta (a_1(n_1) \cdots a_m(n_m))=\sum_{i=1}^m ( -n_i-1)$.
Then $A$ admits a $(Q_B\times \mathbb{C})$-graded algebra structure $A=\bigoplus_{\alpha\in Q_B,d\in\mathbb{C}}A(\alpha,d)$, where 
$A(\alpha,d)=\langle f\in A|\gamma(f)=\alpha,\delta (f)=d\rangle_{k}$.

Let us define a filtration $\{\mathrm{F}_p (Q_B\times\mathbb{C})\}_{p\in\mathbb{Z}}$ on the abelian group
 $Q_B\times \mathbb{C}$ by
\[
\mathrm{F}_p (Q_B\times\mathbb{C})=\{(\alpha,d)\in Q_B\times\mathbb{C}|\alpha\in Q_B,d\in \mathbb{C}, \mathrm{Re}(d)\leq p\}.
\]
Define a filtration on $A$ by $\mathrm{F}_p A=\bigoplus_{(\alpha,d)\in \mathrm{F}_p(Q_B\times \mathbb{C})}A(\alpha,d)$ ($p\in\mathbb{Z}$).
Set $\mathrm{I}_n (A(\alpha,d))=A(\alpha,d)\cap (A\cdot \mathrm{F}_{-n-1} A)$ 
($n\in\mathbb{Z}_+$, $(\alpha,d)\in Q_B\times \mathbb{C}$) and
consider the standard filtered componentwise topology on $A$ and
 denote the standard filtered componentwise completion by $\hat{A}$.

The operator $T:A\rightarrow A$ is continuous, since 
$T(\mathrm{I}_n(A(\alpha,d)))\subset \mathrm{I}_{n-1}(A(\alpha,d+1))$ for $n\geq 1$.

Let $v$ be an element of $\hat{A}$.
Then $v$ has the form $v=\prod_n (v_n+\mathrm{I}_n(A(\alpha,d)))$
with $v_n\in A(\alpha,d)$ ($n\in \mathbb{Z}_+$).
Set $T(v)= \prod_{n=0}^\infty (T(v_{n+1})+\mathrm{I}_{n}(A(\alpha,d+1)))$.
Since $T(\mathrm{I}_n(A(\alpha,d)))\subset \mathrm{I}_{n-1}(A(\alpha,d+1))$, the operator $T$ is well-defined.

\begin{lem}
The vector $T(v)$ belongs to $\hat{A}$.
\end{lem}

The proof is straightforward. Thus we have the continuous operator $T:\hat{A}\rightarrow \hat{A}$.

We define the map $\iota:A\rightarrow \hat{A}$ by $\iota(v)=\prod_n (v+\mathrm{I}_n(A(\alpha,d)))$ for $v\in A(\alpha,d)$ and extend it linearly on $A$.

\begin{lem}
The map $\iota$ is injective.
\end{lem}

Before proving the lemma, we introduce some notations.
Let $\alpha$ be an element of $Q_B$.
Since $Q_B$ is freely generated by $B$, the element $\alpha$ has the form $\sum_{b\in B} n_b b$, where $n_b$ are non-negative integers and equal to zero except for finite $b\in B$.
We call $\sum_{n\in B} n_b$ the {\it height} of $\alpha$ and denote it by $\mathrm{ht}(\alpha)$.
Let $v=r\cdot a_1(n_1) \cdots a_m(n_m)$ be a monomial in $A$ with $r\in \mathbb{C}^\times$.
 For $i=1,\ldots,m$, we set $v_i=a_i(n_i) \cdots a_m(n_m)$.

\begin{proof}
Let $(\alpha,d)$ be an element of $Q_B\times \mathbb{C}$.
It suffices to show that the decreasing sequnce $\{\mathrm{I}_n(A(\alpha,d))\}_n$ is separated.
Put $X=\bigcap_{n=0}^\infty \mathrm{I}_n(A(\alpha,d))$.
Let $v$ be a non-zero element of $A(\alpha,d)$.
Then $v$ has the form $v=\sum_{i=1}^l v^{(i)}$ with $l\geq 1$ and non-zero monomials $v^{(1)},\ldots,v^{(l)}$.
Set $S=
\{(v^{(i)})_j|i=1,\ldots,l, j=1,\ldots, \mathrm{ht}(\alpha)\}
$.
Set $M=\min\{\mathrm{Re}(\delta(s))|s\in S\}$.
Set $N=\max\{0,-\lfloor M \rfloor\}$.
Then, for any $n\geq N$, we have $v\not\in \mathrm{I}_n(A(\alpha,d))$.
Thus, $X=0$, which completes the proof.
\end{proof}

We embed $B$ into $\hat{A}$ by $b\mapsto b_{-1} \mapsto \iota(b_{-1})$.

We define the left action of $A$ over $\hat{A}$ by
\[
A\hookrightarrow \hat{A}\longrightarrow \mathrm{End}(\hat{A}),
\]
where the second map is the left multiplication.

Let $I\subset\hat{A}$ denote the two-sided ideal generated by the locality relations $l(a,b,s,t)$
for all $a,b\in B$ and $s,t\in \mathbb{C}$.
It is a $Q_B\times \mathbb{C}$-graded subspace of $\hat{A}$.
Let $\hat{I}$ denote the componentwise closure of $I$.
Set  $U(B,\mathcal{N},\eta)=\hat{A}/\hat{I}$.
We call $U=U(B,\mathcal{N},\eta)$ a {\it free current algebra (free conformal algebra)}.
Let $J$ denote the left ideal of $\hat{A}$ generated by:
\begin{enumerate}
\item $a(n)$  ($a\in B$, $n\in \mathbb{C}\setminus \mathbb{Z}_{<0}$);
\item $a_1(n_1) \cdots a_m(n_m)$ ($m\geq 2$, $a_1,\ldots,a_m \in B$, 
$n_1,\ldots,n_m\in \mathbb{C}$ with $n_1\not\in \mathcal{N}(a_1,\sum_{i=1}^m a_i)+\mathbb{Z}$).
\end{enumerate}
It is a $Q_B\times \mathbb{C}$-graded subspace of $\hat{A}$.
Let $\hat{J}$ denote the componentwise closure of $J$.
Set $\hat{I}(\alpha,d)=\hat{I}\cap \hat{A}(\alpha,d)$ and $\hat{J}(\alpha,d)=\hat{J}\cap \hat{A}(\alpha,d)$.

Set $F(B,\mathcal{N},\eta)=U/((\hat{I}+\hat{J})/\hat{I})\cong \hat{A}/(\hat{I}+\hat{J})$.
Note that $F=F(B,\mathcal{N},\eta)$ has the induced grading
 $F=\bigoplus_{\alpha\in Q_B, d\in \mathbb{C}}F(\alpha,d)$.
 Since $T(I)\subset I$, $T(J)\subset J$ and $T$ is continuous, we have $T(\hat{I})\subset \hat{I}$ and $T(\hat{J})\subset \hat{J}$.
Hence $T$ induces the operators $T:U\rightarrow U$ and $T:F\rightarrow F$.
Since $\hat{I}$ and $\hat{J}$ are ideal, the left action of $A$ over $\hat{A}$ induces a left $A$-module structure on $U$ and $F$.

Since $I$ is generated by the vectors which is the sum of the monomials of the length two, $\langle a(-1)|a\in B\rangle_\mathbf{k}$ is embedded in $U$.
Consider $X\subset A=\mathbf{k}[X]$.
We have $X\cap (\hat{I}+\hat{J})=\{a(n)|a\in B, n\in \mathbb{C}\setminus \mathbb{Z}_{<0}\}$.
Therefore, $\langle a(-1)|a\in B\rangle_\mathbf{k}$ is embedded in $F$.
Thus we have embedding $B\hookrightarrow F$.

Let $a$ be an element of $B\subset F$.
Set $a(z)=\sum_{n\in\mathbb{C}}a(n) z^{-n-1} \in (\mathrm{End}(F))[[z,z^{\mathbb{C}}]]$.

\begin{prop}\label{sec:propazfield}
The formal power series $a(z)$ is a field on $F$.
\end{prop}

To prove the proposition, we need a lemma.
Let $\alpha$ be an element of $Q_B$.
Set $F(\alpha)=\bigoplus_{d\in \mathbb{C}}F(\alpha,d)$.
\begin{lem}\label{sec:lembounded}
The subspace $F(\alpha)$ is bounded below, that is, there exists $N\in\mathbb{R}$ such that
\[
F(\alpha)=\bigoplus_{d\in \mathbb{C},\mathrm{Re}(d)\geq N} F(\alpha,d).
\]
\end{lem}

\begin{proof}
We show that there exists $N\in\mathbb{R}$ such that for any $d\in\mathbb{C}$ with $\mathrm{Re}(d)\leq N$,
we have $\hat{A}(\alpha,d)\subset \hat{I}(\alpha,d)+\hat{J}(\alpha,d)$.
Put $m=\mathrm{ht}(\alpha)$.
Let $S$ be the set consisting of the vectors of the form 
\[
a_1(\mathcal{N}(a_1,{\textstyle \sum_{j=2}^m} a_j )-1)\cdots a_i(\mathcal{N}(a_i,{\textstyle \sum_{j=i+1}^m} a_j)-1)\cdots  a_m(\mathcal{N}(a_m,0)-1)
\]
 with $a_1,\ldots,a_m\in B$ with $a_1+\cdots+a_m=\alpha$.
This is a finite set.
Set $M=\min \{\mathrm{Re}(\delta(s))|s\in S\}$.
Then, for $d\in\mathbb{C}$ with $\mathrm{Re}(d)<M$, we have $\hat{A}(\alpha,d)\subset \hat{I}(\alpha,d)+\hat{J}(\alpha,d)$,
which completes the proof.
\end{proof}

Now we show that $a(z)$ is a field.

\begin{proof}[Proof of Proposition \ref{sec:propazfield}]
Let $\alpha$ be an element of $Q_B$ and $n$ an element of $\mathbb{C}\setminus (\mathcal{N}(e(a),\alpha)+\mathbb{Z})$.
Then, for any homogeneous $v$ with charge $\alpha$, we have $a(n) v=0$.
Let $(\beta,e)$ be an element of $Q_B\times \mathbb{C}$.
Let $v$ be an element of $F(\beta,e)$.
We show that $a(n) v=0$ for $n\in \mathcal{N}(e(a),\beta)+\mathbb{Z}$ with $\mathrm{Re}(n)\gg 0$.
By Lemma \ref{sec:lembounded}, the subspace $F(e(a)+\beta)$ has the form
 $\bigoplus_{d\in\mathbb{C}, \mathrm{Re}(d)\geq N}F(e(a)+\beta,d)$
with $N\in \mathbb{Z}$.
Since $a(n) v\in F(e(a)+\beta,-n-1+e)$,
 we have $a(n) v=0$ for $n$ with $\mathrm{Re}(n)>N+1-e$.
Thus, $a(z)$ is a field on $F$.
\end{proof}

We show that $F$ is algebraically-generated by $\{a(n)|a\in B, n\in \mathbb{C}\}$.

\begin{prop}
For any $v\in F(\alpha,d)$, there exists $u\in A(\alpha,d)$ such that
$v=\prod_n (u+\mathrm{I}_n(A(\alpha,d))+\hat{I}(\alpha,d)+\hat{J}(\alpha,d)$.
\end{prop}

\begin{proof}
Let $v$ be an element of $F(\alpha,d)$.
There exists $w\in \hat{A}(\alpha,d)$ such that $v=w+\hat{I}(\alpha,d)+\hat{J}(\alpha,d)$.
The vector $w$ has the form $w=\prod_n (w_n+I_n(A(\alpha,d)))$ with $w_n\in A(\alpha,d)$ and
for $n_1\leq n_2$, $w_{n_1}+I_{n_1}(A(\alpha,d))=w_{n_2}+I_{n_1}(A(\alpha,d))$.
Put $m=\mathrm{ht}(\alpha)$.
Let $S$ be the set consisting of the vectors of the form 
\[
{\textstyle 
a_1(\mathcal{N}(a_1,\sum_{j=2}^l a_j )-1)\cdots a_i(\mathcal{N}(a_i,\sum_{j=i+1}^l a_j)-1)\cdots a_l(\mathcal{N}(a_l,0)-1)
}
\]
 with $l=1,\ldots,m$, $a_1,\ldots,a_l\in B$ with $\alpha-a_1-\cdots-a_l\in (Q_0)_+$.
This is a finite set.
Set $M=\min \{\mathrm{wt}(s)|s\in S\}$.
Set $N=\max\{-\lfloor M \rfloor,0\}$.
Then for any $n\geq N$, we have $\mathrm{I}_n(A(\alpha,d))\subset \hat{I}(\alpha,d)+ \hat{J}(\alpha,d)$.
Therefore $w=\prod_n (w_N+\mathrm{I}_n(A(\alpha,d)))+\hat{I}(\alpha,d)+\hat{J}(\alpha,d)$.
Thus, we have the proposition.
\end{proof}

Therefore, the space $F$ is generated by $a(z)$ ($a\in B$).

Consider the linear map $T$ on $F$.
The field $a(z)$ is translation covariant with respect to $T$.
By Theorem \ref{sec:thmbk} (Reconstrunction Theorem), we obtain the generalized vertex algebra $(F(B,\mathcal{N},\eta),Y,|0\rangle,T)$ with a charge factor $(Q_B,\Delta,\eta)$, where
$\Delta(\alpha,\beta)=\mathcal{N}(\alpha,\beta)+\mathbb{Z}$.

By the construction, we have the following universality of $F$.

\begin{thm}\label{sec:thmuniv}
Let $V'$ be a GVA with a charge factor $(Q',\Delta',\eta')$ on $V'$.
Suppose that there exists a group homomorphism $\rho:Q_B\rightarrow Q'$ such that $\Delta(\alpha,\beta)=\Delta'(\rho(\alpha),\rho(\beta))$ and $\eta(\alpha,\beta)=\eta'(\rho(\alpha),\rho(\beta))$ ($\alpha,\beta\in Q_B$).
Suppose also that there exists a map $i:B\rightarrow V'$ such that 
 each vector $i(b)$ is zero or a homogeneous vector of charge $\rho(b)$ ($b\in B$),
that $i(B)$ generates the GVA $V'$
and that the map  $\mathcal{N}':i(B)\times i(B)\rightarrow \mathbb{C}$ with $\mathcal{N}'(i(b),i(c))=\mathcal{N}(b,c)$ is well-defined and is a locality bound on $i(B)$.
Then, there is the unique surjective GVA-homomorphism $\pi:F\rightarrow V'$ such that $\pi|_B=i$.
\end{thm}

We call $F$ a {\it free} generalized vertex algebra.
This is a generalization of the free vertex (super)algebras \cite{R}.

\subsection{Universality of the free generalized vertex algebras}

Let $V$ be a generalized vertex algebra with a charge factor $(Q,\Delta,\eta)$.
Let us denote the charge of a $Q$-homogeneous element $v$ by $\gamma(v)$.
Let $B$ be a totally ordered set with a map $i:B\rightarrow V$.
Suppose that any element of $i(B)$ is a $Q$-homogeneous vector.
Suppose that the set $i(B)$ generates $V$.
Let $Q_B=\bigoplus_{v\in B}\mathbb{Z}e(v)$ be a free abelian group with a $\mathbb{Z}$-basis $\{e(v)\}_{v\in B}$ indexed by $B$.
Let $\mathcal{N}:B\times B\rightarrow \mathbb{C}$ be a locality bound on $(B,i)$.

Now we define another function $\mathcal{N}':B\times B\rightarrow \mathbb{C}$.
For $a,b\in B$ with $a\neq b$, set $\mathcal{N}'(a,b)=\mathcal{N}(a,b)$.
Let $a$ be an element of $B$.
We denote $i(a)$ and $\gamma(i(a))$ by $a$.
If $a(\mathcal{N}(a,a)-1)a= 0$ and $\eta(a,a)\neq e^{-\pi i \mathcal{N}(a,a)}$, then set $\mathcal{N}'(a,a)=\mathcal{N}(a,a)-1$;
otherwise, set $\mathcal{N}'(a,a)=\mathcal{N}(a,a)$.
Then, $\mathcal{N}'$ is again a locality bound on $(B,i)$.

\begin{lem} For any $a\in B$,
\[
\eta(a,a)=e^{-\pi i \mathcal{N}'(a,a)}.
\]
\end{lem}

\begin{proof}
Let $a$ be an element of $B$.
Since $\eta(\alpha,\beta)\eta(\beta,\alpha)=e^{-2\pi i \Delta(\alpha,\beta)}$,
we have $\eta(a,a)=\pm e^{-\pi i \mathcal{N}(a,a)}$.
Suppose $a(\mathcal{N}(a,a)-1)a=0$ and $\eta(a,a)\neq e^{-\pi i \mathcal{N}(a,a)}$.
Then, $\eta(a,a)= -e^{-\pi i \mathcal{N}(a,a)}$.
Therefore, $\eta(a,a)= e^{-\pi i (\mathcal{N}(a,a)-1)}=e^{-\pi i \mathcal{N}'(a,a)}$.
Suppose $a(\mathcal{N}(a,a)-1)a\neq 0$.
Then, $\mathcal{N}'(a,a)=\mathcal{N}(a,a)$.
By eq.\ (\ref{eqn:local2}), we have
\begin{eqnarray*}
&&\sum_{j\in \mathbb{Z}_+}(-1)^j \begin{pmatrix}\mathcal{N}'(a,a) \\ j \end{pmatrix} \biggl(a(-1+\mathcal{N}'(a,a)-j)a(-1+j) \\ 
&&\quad-\eta(a,a)e^{\pi i \mathcal{N}'(a,a)} a(-1+\mathcal{N}'(a,a)-j) a(-1+j)\biggr)|0\rangle=0.
\end{eqnarray*}
Since $a(-1)|0\rangle=a$ and $a(-1+j)|0\rangle=0$ for $j>0$, we have $(1-\eta(a,a)e^{\pi i \mathcal{N}'(a,a)})a(-1+\mathcal{N}'(a,a))a=0$.
By the assumption, we have $1-\eta(a,a)e^{\pi i \mathcal{N}'(a,a)}=0$, which completes the proof.
\end{proof}

\begin{cor}[Corollary of Theorem \ref{sec:thmuniv}]\label{sec:propobject}
There exist a bimultiplicative
function $\eta':Q_B\times Q_B\rightarrow \mathbb{C}^\times$ with $\eta'(\alpha,\alpha)=e^{-\pi i \mathcal{N}'(\alpha,\alpha)}$
 such that there exists a surjective GVA homomorphism $\pi:F(B,\mathcal{N}',\eta')\rightarrow V$ with
$\pi|_B=i$.
\end{cor}

\begin{proof}
We have the canonical group homomorphism $\pi:Q_B\rightarrow Q$ by $\pi(\sum_{v\in B} a_v e(v))=\sum a_v \gamma(i(v))$.
Let $\alpha,\beta$ be elements of $Q_B$.
Define $\eta':Q_B\times Q_B\rightarrow \mathbb{C}^\times$ by $\eta'(\alpha,\beta)=\eta(\pi(\alpha),\pi(\beta))$.
Then, we have $\eta'(\alpha,\alpha)=e^{-\pi i \mathcal{N}'(\alpha,\alpha)}$ for any $\alpha\in Q_B$.
Then, by Theorem \ref{sec:thmuniv}, we have the proposition.
\end{proof}

\section{The generalized lattice vertex algebras and the free generalized vertex algebras}

\subsection{The generalized principal subspaces}

We introduce a notion equivalent to the free generalized vertex algebras---the {\it generalized principal subspaces}. 

Let $\mathfrak{h}$ be a vector space over $\mathbb{C}$ equipped with a bilinear form
$(\cdot|\cdot): \mathfrak{h}\times \mathfrak{h} \rightarrow \mathbb{C}$.
Let $\varepsilon: \mathfrak{h}\times \mathfrak{h}\rightarrow \mathbb{C}^\times$
be a $2$-cocycle.
Let $L\subset \mathfrak{h}$ be a lattice equipped with 
a $\mathbb{Z}$-basis $B\subset \mathfrak{h}$.
Assume that $B$ is also a $\mathbb{C}$-basis of $\mathfrak{h}$, so that
$\mathfrak{h}\cong \mathbb{C}\otimes_{\mathbb{Z}}L$.
Consider the generalized vertex algebra
\[
V_\mathfrak{h}=M(1) \otimes_{\mathbb{C}} \mathbb{C}[\mathfrak{h}]
\]
associated to the vector space $\mathfrak{h}$ \cite{BK}.
Consider the lattice generalized vertex algebra $V_L\subset V_\mathfrak{h}$ associated to the lattice $L$ (cf. \cite{DL, BK}).
Consider the $\varepsilon$-modified GVA $V_L^\varepsilon$.

\begin{dfn}
The {\it generalized principal subspace} (generalized principal subalgebra) of the lattice GVA $V_L^\varepsilon$
is the generalized vertex subalgebra
\[
W_L^\varepsilon(B)=\langle e^{\beta} | \beta\in B \rangle
\]
generated by the vectors $\{e^{\beta}|\beta\in B\}$.
\end{dfn}

We denote $W_L=W_L^\varepsilon=W_L^\varepsilon(B)$.

\begin{rem}
When $L$ is an integral lattice, $W_L$ agrees with the {\it principal subalgebra} of $V_L$ introduced in \cite{MP}.
When $L$ is the $A_l$ root lattice and the basis $B$ is a base of roots, $W_{A_l}$ agrees with the {\it principal subspace} of the basic representation of affine Kac-Moody Lie algebra $A_l^{(1)}$ introduced in \cite{SF}.
\end{rem}

Let $L'$ be an abelian subgroup of $\mathfrak{h}$ containing $L$ with 
a $2$-cocycle $\varepsilon:L'\times L'\rightarrow \mathbb{C}^\times$ extending the $2$-cocycle  $\varepsilon$ on $L$.
Consider the GVA $V_{L'}=\bigoplus_{\alpha\in L'} M(1)\otimes e^\alpha$ and the $\varepsilon$-modified GVA
$V_{L'}^\varepsilon$.
Let $\lambda$ be an element of $L'$.

\begin{dfn}
The {\it generalized principal subspace} of weight $\lambda$ over $W_L$ with the $\varepsilon$ on $L'$
is the cyclic $W_L$-submodule
\[
W_L^\varepsilon(B;\lambda)= W_L\cdot e^{\lambda} \subset V_{L'}.
\]
\end{dfn}

We denote $W_L(\lambda)=W_L^\varepsilon(\lambda)=W_L^\varepsilon(B;\lambda)$.
Note that $W_L(0)=W_L$ as vector subspaces.

Consider the vector $v=h_k(-i_k) \ldots h_1(-i_1) e^\alpha$
 with $k\geq 0$,
$h_1,\ldots,h_k \in \mathfrak{h}$, $i_1,\ldots,i_k \geq 1$ and $\alpha\in \mathfrak{h}$.
The {\it weight} of $v$ is
\[
\mathrm{wt}(v)= i_k+\cdots +i_1 + \frac{(\alpha|\alpha)}{2}.
\]
Let $V$ be a vector subspace of $(V_{L'})^\varepsilon$.
Set
$V_n=\mathbb{C}\mbox{-span}\{v\in V| \mathrm{wt}(v)=n\}$.
The grading $V_{L'}^\varepsilon=\bigoplus_{n\in \mathbb{C}} (V_{L'}^\varepsilon)_n$ is a weight-grading.
The subspaces $V_L^\varepsilon$, $W_L^\varepsilon$ and $W_L^\varepsilon(\lambda)$ are weight-graded.
Actually, $W_L^\varepsilon$ is $L_+$-graded, where $L_+=\bigoplus_{\alpha\in B} \mathbb{Z}_+ \alpha$.

\subsection{Bases of the generalized principal subspaces}

Set $\mathcal{N}(e^\alpha,e^\beta)=-(\alpha|\beta)$ for $\alpha,\beta\in B$.
Identify the $\mathbb{Z}$-basis $B$ and the subset $B=\{e^\alpha|\alpha\in B\}\subset W_L^\varepsilon (B)$.
The symmetric function $\mathcal{N}$ is a locality bound on the subset $B$.
Consider a total order $\leq$ on $B$.

\begin{thm}\label{sec:basis}
The set $\mathcal{C}(W_L^\varepsilon(B),B,\mathcal{N})$ is a basis of $W_L^\varepsilon(B)$.
\end{thm}

Our proof is a generalization of the one given in \cite{G}.
We first prove the case when $\varepsilon$ is trivial, that is, $\varepsilon(\alpha,\beta)=1$ for any $\alpha,\beta\in L$.

\begin{prop}\label{sec:basistrivial}
If the $2$-cocycle $\varepsilon$ is trivial, then the set $\mathcal{C}(W_L^\varepsilon(B),B,\mathcal{N})$ is a basis of $W_L^\varepsilon(B)$.
\end{prop}

Suppose the $2$-cocycle $\varepsilon$ is trivial.
Consider $L'=\mathfrak{h}$ with the trivial $2$-cocycle $\varepsilon$ on $\mathfrak{h}$.
We denote $W=W_L^\varepsilon(B)$ and $W(\lambda)=W_L^\varepsilon(B;\lambda)$ ($\lambda\in \mathfrak{h}$).

First, we extend the metric vector space $\mathfrak{h}$, which is not necessarily non-degenerate, to a non-degenerate one.
Consider the kernel  $K=\ker (\cdot|\cdot)$ of the bilinear form $(\cdot|\cdot)$ of $\mathfrak{h}$.
Let $C$ be a basis of $K$.
Let $K'$ be a complement of $K$ in $\mathfrak{h}$.
Let $K^*$ be a vector space isomorphic to $K$ with a basis $C^*$ and a bijection $*:C\rightarrow C^*$, $\mu \mapsto \mu^*$.
Set $\mathfrak{h}'=\mathfrak{h}\oplus K^*$ and
extend the bilinear form to be
$(K'|K^*)=0$ and $(\mu|\nu^*)=\delta_{\mu,\nu}$
for $\mu,\nu\in C$.
Then the bilinear form is non-degenerate.
Let $D$ denote the dual basis of the basis $B\sqcup C^*$
with respect to $(\cdot|\cdot)$ with a bijection $\circ:B\sqcup C^*\rightarrow D$, $\beta\mapsto \beta^\circ$, with 
$(\beta^\circ|\gamma)=\delta_{\beta,\gamma}$ ($\beta,\gamma\in B\sqcup C^*$).
Consider the generalized vertex algebra $V_{\mathfrak{h}'}$ (with the trivial $2$-cocycle on $\mathfrak{h}'$).
Then, $V_{\mathfrak{h}}$ is a generalized vertex subalgebra of $V_{\mathfrak{h}'}$.

Let $\lambda$ be an element of $\mathfrak{h}'$.
Consider the simple current operator
\[
e_\lambda:W\longrightarrow W(\lambda),\quad e_\lambda(h\otimes e^\alpha)=h\otimes e^{\alpha+\lambda},
\]
for $h\in M(1)$ and $\alpha\in L_+$.
Then, we have
\begin{equation}\label{eqn:comms}
e_{\lambda} \circ e^{\alpha}(m)= e^{\alpha}(m-(\lambda|\alpha)) \circ e_{\lambda}
\end{equation}
for $\alpha\in L_+$ and $m\in \mathbb{C}$.
Let $\beta,\gamma$ be elements of $B$.
Note that $\mathcal{N}(e^{\beta^\circ},e^\gamma)=-\delta_{\beta,\gamma}\leq 0$.
In formula (\ref{eqn:local2}), substitute $a=e^{\beta^\circ}$, $b=e^{\gamma}$ and $N=0$.
Then,  we obtain the commutation relation
\begin{eqnarray}\label{eqn:commutation}
e^{\beta^\circ}(s) e^{\gamma}(t) -(-1)^{\delta_{\beta,\gamma}} e^{\gamma}(t) e^{\beta^\circ}(s)=0.
\end{eqnarray}

Let $a$ be an element of $\mathcal{C}(W,B,\mathcal{N})$.
Then $a$ has the form
\[
 a=\Psi(e^{\beta_1},m_1;\cdots;e^{\beta_k},m_k)=\prod_{i=1}^k \left( e^{\beta_i}(m_i-
{\textstyle \sum_{j=1}^{i-1}}(\beta_i|\beta_j))\right)|0\rangle
\]
with a non-negative integer $k$, elements $\beta_1,\ldots,\beta_k\in B$ with $\beta_1\leq \cdots\leq \beta_k$ and
negative integers $m_1,\ldots,m_k<0$ such that $m_i\geq m_j$ 
if $i<j$ and $\beta_i=\beta_j$. 
Let $i$ be an element of $\{1,\ldots,k\}$.
If $i\geq 2$ and $\beta_i=\beta_{i-1}$, then set
\[
D_a^i=e_{-\beta_i}\circ\left( e_{-(\beta_i)^\circ}\circ (e^{(\beta_i)^\circ})(-1)\right)^{-m_i+m_{i-1}};
\]
otherwise, set
\[
D_a^i=e_{-\beta_i}\circ\left( e_{-(\beta_i)^\circ}\circ (e^{(\beta_i)^\circ})(-1)\right)^{-m_i-1}.
\]
Note that for $\gamma\in B$ and $m\in \mathbb{C}$, we have the commutation relation
\begin{eqnarray}\label{eqn:commd}
&&(e_{-(\beta_i)^\circ}\circ(e^{(\beta_i)^\circ})(-1)) \circ e^\gamma(m) \nonumber\\
&&\quad\quad =(-1)^{\delta_{\beta_i,\gamma}}e^\gamma(m+\delta_{\beta_i,\gamma})\circ
(e_{-(\beta_i)^\circ}\circ(e^{(\beta_i)^\circ})(-1)),
\end{eqnarray}
by eq.\,(\ref{eqn:comms}) and (\ref{eqn:commutation}).
Consider the operator
\[
X_a=\prod_{i=1}^k D_a^i=D_a^k\circ \cdots \circ D_a^1.
\]

\begin{lem}
There is a non-zero scalar $c\in \mathbb{C}^\times$ such that
\begin{equation}\label{eqn:xaa}
X_a(a)=c\cdot |0\rangle.
\end{equation}
\end{lem}

\begin{proof}
We prove (\ref{eqn:xaa}) by induction on $k$.

Assume $k=0$. Then, $a=|0\rangle$, and $X_a$ is the identity operator on $W$.
Therefore, $X_a(a)=c|0\rangle$ with $c=1\in \mathbb{C}^\times$.

Assume $k>0$.
Consider the operator $D_a^{1}=e_{-\beta_1}\circ (e_{-(\beta_1)^\circ}\circ (e^{(\beta_1)^\circ})(-1))^{-m_1-1}$.
Let $k_1$ denote the maximum number in $\{1,\ldots,k\}$ such that $\beta_1=\cdots=\beta_{k_1}$.
Since 
\begin{equation}\label{eqn:vacuum}
e^{\alpha}(-1)|0\rangle=e^{\alpha} \quad \mbox{ and } \quad e_{-\alpha}e^{\alpha}=|0\rangle \quad \mbox{ for any } \alpha\in \mathfrak{h}',
\end{equation}
by eq.\,(\ref{eqn:commd}), we have
\begin{eqnarray*}
D_a^{1}(a)&=&  (-1)^{(-m_1-1)k_1} \cdot e_{-\beta_1} \Psi(e^{\beta_1},t_1;\cdots;e^{\beta_k},t_k)\\
&=& (-1)^{(-m_1-1)k_1} \cdot e_{-\beta_1} \circ \prod_{i=1}^k \left( e^{\beta_i}(t_i-{\textstyle \sum_{j=1}^{i-1}} (\beta_i|\beta_j))\right)|0\rangle
\end{eqnarray*}
with $t_i=m_i-m_1-1$ for $i=1,\ldots,k_1$ and $t_j=m_j$ for $j=k_1+1,\ldots,k$.
Since $t_1=m_1-m_1-1=-1$, by eq.\,(\ref{eqn:vacuum}) and (\ref{eqn:comms}),
we have
$D_a^{1}(a)=(-1)^{(-m_1-1)k_1} \cdot \Psi(e^{\beta_2},t_2;\cdots;e^{\beta_k},t_k)$.
Then, $X_a=X_{a'}\circ D_a^{1}$, since $-t_i+t_{i-1}=-m_{i}+m_{i-1}$ for $i=2,\ldots,k_1$.
By the induction hypothesis, we have $X_{a'}(a')=c'|0\rangle$ with $c'\in \mathbb{C}^\times$, so that
 $X_a(a)=c|0\rangle$ with $c=(-1)^{(-m_1-1)k_1}c'$, as required.
\end{proof}

Let $b$ be a monomial in $\mathcal{C}(W,B,\mathcal{N})$.
Then $b$ has the form
\[
 b=\Psi(e^{\gamma_1},n_1;\cdots;e^{\gamma_l},n_l)=\prod_{i=1}^l \left( e^{\gamma_i}(n_i-
{\textstyle \sum_{j=1}^{i-1}}(\gamma_i|\gamma_j))\right)|0\rangle
\]
with a non-negative integer $l\geq 0$, elements $\gamma_1,\ldots,\gamma_l\in B$ with $\gamma_1\leq \cdots \leq \gamma_l$ and
negative integers $n_1,\ldots,n_l<0$ such that 
$n_i\geq n_j$ if $i<j$ and $\gamma_i=\gamma_j$.
We say
\[\\
a\succ b,
\]
if {\bf (1)} $k=l$; {\bf (2)} $\beta_i=\gamma_i$ for all $i=1,\ldots,k$,
and
{\bf (3)} there exists $t\in \{1,\ldots,k\}$ such that
$m_t<n_t$ and $m_j=n_j$ for any $j<t$.
Note that when $a$ and $b$ have the same charge, either $a\prec b$, $a=b$ or $a\succ b$ hold.

\begin{lem}
\label{sec:lemdaisyou}
If $a\succ b$, then $X_a(b)=0$.
\end{lem}

\begin{proof}
Suppose $a\succ b$ with an element $t\in \{1,\ldots,k\}$ such that $m_t<n_t$ and $m_j=n_j$ for $j<t$.
Put
$X_a'=\prod_{i=1}^{t-1} D_a^{i}$.
If $t>1$ and $\beta_{t-1}=\beta_{t}$, then set $m:=m_{t-1}$; otherwise, set $m:=-1$.
Then, $X_a'(b)$ is a multiple of
$
\Psi(e^{\beta_t},n_t-m-1;e^{\beta_{t+1}},p_{t+1};\cdots;e^{\beta_k},p_k)
$
with some $p_{t+1},\ldots,p_k$.
Therefore,  $D_a^t X_a'(b)=0$, since $n_t$ is greater than $m_t$ and 
$e^\gamma(-1+j)|0\rangle=0$ for $j\geq 1$.
Hence,  $X_a(b)=0$, as desired.
\end{proof}

\begin{proof}[Proof of Proposition \ref{sec:basistrivial}]
By Corollary \ref{sec:span}, it suffices to show that the set $\mathcal{C}{(W,B,\mathcal{N})}$ is linearly independent.
Let $m$ be a positive integer.
Let $v_1,\ldots,v_m$ be elements of $\mathcal{C}(W,B,\mathcal{N})$ with
$v_i\neq v_j$ ($i,j\in \{1,\ldots,m$\}, $i\neq j$) and $v_1\not \prec v_i$ ($i=2,\ldots,m$).
Let $c_1,\ldots,c_m$ be complex numbers
with 
\begin{equation}\label{eqn:indep}
c_1v_1+\cdots+c_m v_m=0.
\end{equation}
We show $c_1=\cdots=c_m=0$ by induction on $m$.
When $m=1$, we have $c_1=0$.
Therefore, assume $m>1$.
Note that the elements of $\mathcal{C}(W,B,\mathcal{N})$ are charge-homogeneous vectors.
If the charge of $v_1$ is not equal to the charge of $v_i$ for some $i$, by the induction hypothesis,
we have $c_1=\cdots=c_m=0$.
Assume that the charge of $v_1$ is equal to the charge of $v_i$ for all $i=1,\ldots,m$.
Then, we have $v_1\succ v_i$ for $i=2,\ldots,m$.
By applying the operator $X_1$ to both sides of (\ref{eqn:indep}),
we obtain $c_1\cdot c\cdot |0\rangle=0$ with $c\in \mathbb{C}^\times$,
by Lemma \ref{sec:lemdaisyou} and equality (\ref{eqn:xaa}).
Therefore, $c_1=0$.
By the induction hypothesis, $c_2=\cdots=c_m=0$, which completes the proof.
\end{proof}

\begin{proof}[Proof of Theorem \ref{sec:basis}.]
Let $a$ be an element of the subset $B$.
Note that the $\varepsilon$-modified GVA $V_L^\varepsilon$ is a GVA with the same underlying vector space $V_L$ and the modified vertex operators $Y^\varepsilon(a,z)$.
The modified vertex operator has the form $Y^\varepsilon(a,z)=\sum a^\varepsilon(n) z^{-n-1}$ with operators $a^\varepsilon(n):V_L\rightarrow V_L$ ($n\in\mathbb{C}$).
For a homogeneous $v\in V_L$, we see that $a^\varepsilon(n) v$ is a non-zero multiple of $a(n) v$,
say, $\varepsilon(\gamma(a),\gamma(v))$.
Here, $\gamma(a)$ and $\gamma(v)$ are charges of $a$ and $v$.
Since the elements of $\mathcal{C}(W,B,\mathcal{N})$ are monimials in the homogeneous vectors $B$,
and the set $\mathcal{C}(W,B,\mathcal{N})$ is linearly independent, we see that
the set $\mathcal{C}(W_L^\varepsilon(B),B,\mathcal{N})$ is linearly independent.
By Corollary \ref{sec:span}, we have the theorem.
\end{proof}

Let $L'$ be an abelian subgroup of $\mathfrak{h'}$ containing $L$ with a $2$-cocycle $\varepsilon$ on $L'$
extending the $2$-cocycle $\varepsilon$ on $L$.
Let $\lambda$ be an element of  $L'$.
Put $W(\lambda)=W_L^\varepsilon(B;\lambda)$.
Consider the subset $\mathcal{C}(W(\lambda),B,\mathcal{N};\lambda)\subset W(\lambda)$ consisting of the elements
\[
\Psi(e^\lambda,-1;e^{\beta_1},m_1;\cdots;e^{\beta_k},m_k)=\prod_{i=1}^k \left(  e^{\beta_i}(m_i-{\textstyle \sum_{j=1}^{i-1}}(\beta_i|\beta_j)-(\beta_i|\lambda))\right) e^\lambda
\]
with a non-negative integer $k\geq 0$, elements $\beta_1,\ldots,\beta_k\in B$ with $\beta_1\leq \cdots \leq \beta_k$ and negative integers $m_1,\ldots,m_k<0$ such that $m_i\geq m_j$ if $i<j$ and $\beta_i=\beta_j$.
Here, we extend $\mathcal{N}$ to the set $B\cup \{\lambda\}$ by $\mathcal{N}(a,b)=-(a|b)$ ($a,b\in B\cup \{\lambda\}$) and define $\Psi(e^\lambda,m_0;e^{\beta_1},m_1;\cdots;e^{\beta_k},m_k)$ as section \ref{sec:secspan} ($m_0,m_1,\ldots,m_k\in \mathbb{C}$).

\begin{cor}\label{sec:corbasis2}
The set $\mathcal{C}(W(\lambda),B,\mathcal{N};\lambda)$ is a basis of $W(\lambda)$.
\end{cor}

\begin{proof}
By Theorem \ref{sec:basis} and using the simple current operator $e_{\lambda}$, we have the corollary.
\end{proof}

This is the generalization of the results of \cite{MP}.

\subsection{Structure of the free generalized vertex algebras and generalized principal subspaces}
Let $B$ be a set.
Let $Q=Q_B=\bigoplus_{a\in B}\mathbb{Z}e(a)$ be a free abelian group with a $\mathbb{Z}$-basis $\{e(a)\}_{a\in B}$ indexed by $B$.
Let $\mathcal{N}:B\times B\rightarrow \mathbb{C}$ be a symmetric function,
 and extend it bilinearly to $Q\times Q\rightarrow \mathbb{C}$.
Let  $\eta:Q\times Q\rightarrow \mathbb{C}^\times$ be a bimultiplicative function
with  $\eta(\alpha,\alpha)=e^{-\pi i \mathcal{N}(\alpha,\alpha)}$.
Consider the free generalized vertex algebra $(F(B,\mathcal{N},\eta),Q,\Delta,\eta)$,
where $\Delta:Q\times Q\rightarrow \mathbb{C}/\mathbb{Z}$ is defined by $\Delta(\alpha,\beta)=\mathcal{N}(\alpha,\beta)+\mathbb{Z}$.

Define the symmetric bilinear form $(\cdot|\cdot):Q\times Q\rightarrow \mathbb{C}$ by $(\alpha|\beta)=-\mathcal{N}(\alpha,\beta)$.
Set $\mathfrak{h}=Q\otimes_{\mathbb{Z}}\mathbb{C}$.
Extend to $\mathfrak{h}$ the bilinear form $(\cdot|\cdot)$ on $Q$.
Consider the generalized vertex algebra $V_{\mathfrak{h}}$.
Then, $\mathcal{N}(\alpha,\beta)$ is a locality bound of $e^\alpha, e^\beta\in V_\mathfrak{h}$ ($\alpha,\beta\in Q$).

In the sequel, we denote $a=e(a)$ for $a\in B$.

Let $\alpha,\beta$ be elements of $Q$.
Set $\omega(\alpha,\beta)=\eta(\alpha,\beta)e^{\pi i \mathcal{N}(\alpha,\beta)}$.
Then $\omega:Q\times Q\rightarrow \mathbb{C}^\times$ is bimultiplicative.
Since $\eta(\alpha,\alpha)=e^{-\pi i \mathcal{N} (\alpha,\alpha)}$, we have $\omega(\alpha,\alpha)=1$.
By the well-known result, we obtain the $2$-cocycle $\varepsilon:Q\times Q\rightarrow \mathbb{C}^\times$
with $\varepsilon(\alpha,\beta)\varepsilon(\beta,\alpha)^{-1}=\omega(\alpha,\beta)$.
Consider the $\varepsilon$-modified  lattice GVA $V_Q=V_Q^\varepsilon$ and the principal subspace $W_Q^\varepsilon(B)$.
Then, we have the projection $\pi:F(B,\mathcal{N},\eta)\rightarrow W_Q^\varepsilon(B)$ with $\pi(a)=e^a$ for each $a\in B$.

\begin{thm}\label{sec:thmisomorphic}
The free generalized vertex algebra $F(B,\mathcal{N},\eta)$ is isomorphic to the generalized principal subspace $W_Q^\varepsilon(B)$.
\end{thm}

\begin{proof}
Put $F=F(B,\mathcal{N},\eta)$ and $W=W_Q^\varepsilon(B)$.
Consider a total order on $B$ and spanning sets $\mathcal{C}(F,B,\mathcal{N})$ and $\mathcal{C}(W,B,\mathcal{N})$.
Consider the surjection $\pi:F \rightarrow W$.
By the construction, the image of the spanning set $\mathcal{C}(F,B,\mathcal{N})$ agrees with $\mathcal{C}(W,B,\mathcal{N})$.
Since $\mathcal{C}(W,B,\mathcal{N})$ is a basis, we have the inverse map $W\rightarrow F$.
Thus, we have the corollary.
\end{proof}

This is the generalization of the results of \cite{R}.
Note that another kind of PBW-type bases for the free vertex algebras is given in \cite{R}.
(However, we do not generalize the bases in this paper.)

\section{Presentation of lattice vertex algebras in terms of generators and relations}

Let $L$ be a lattice with a symmetric $\mathbb{Z}$-bilinear form $(\cdot|\cdot):L\times L\rightarrow \mathbb{C}$.
Let $\varepsilon:L\times L\rightarrow \mathbb{C}^\times$ be a $2$-cocycle.
Denote by $\omega$ the canonical invariant of $\varepsilon$.
Consider the $\varepsilon$-modified generalized vertex algebra $V_L=(V_L)^\varepsilon$.
Let $\Pi$ be a $\mathbb{Z}$-basis of $L$.

Set $B=\{\widetilde{a}|a\in \Pi\} \sqcup \{\widehat{a}|a\in \pm \Pi\}$.
We set $\widetilde{-a}=-\widetilde{a}$.
Define the map $\delta:B\rightarrow L$ by $\delta(\widehat{a})=a$ ($a\in \pm \Pi$) and $\delta(\widetilde{a})=0$ ($a\in \Pi$).
Let $Q_B$ denote the free abelian group generated by $B$.
Extend $\delta$ bilinearly on $Q_B$ and denote $\delta:Q_B\rightarrow L$.
Define a bimultiplicative function $\eta:Q_B\times Q_B\rightarrow \mathbb{C}$
 by $\eta(\alpha,\beta)=e^{\pi i (\delta(\alpha)|\delta(\beta))}\omega(\delta(\alpha),\delta(\beta))$ for
$\alpha,\beta\in Q_B$.
Define $\mathcal{N}:B\times B\rightarrow \mathbb{C}$ by 
\[
\mathcal{N}(v^a,v^b)=-(a|b),\quad \mathcal{N}(\widetilde{a},\widetilde{b})=2,\quad \mathcal{N}(\widetilde{a},v^b)=1, \quad (a,b\in \pm \Pi).
\]
Consider the free GVA $F=F(B,\mathcal{N},\eta)$.
Then we have the projection $F\rightarrow V_L$.

\begin{lem}\label{sec:lemlattice}
The lattice generalized vertex algebra $V_L$ is presented by generators $\widetilde{a}$ ($a\in \Pi$) 
and $\widehat{a}=:v^{ a}$ for $a\in \pm \Pi$
with the locality factor $\eta$ and locality bounds $\mathcal{N}$,
and the following relations ($a,b\in\pm \Pi$):
\begin{enumerate}
\item $\widetilde{a}(0)\widetilde{b}=0,\quad \widetilde{a}(1)\widetilde{b}=(a|b)|0\rangle$,
\item $\widetilde{a}(0) v^b=(a|b)v^b$,
\item $v^a((a|a)-1) v^{-a}=|0\rangle$,
\item $Tv^a =\widetilde{a}(-1) v^a$.
\end{enumerate}
\end{lem}

Note that $\delta:Q_B\rightarrow L$ induces the bimultiplicative and bilinear functions $\eta:L\times L\rightarrow \mathbb{C}^\times$
and $\Delta:L\times L\rightarrow \mathbb{C}/\mathbb{Z}$ on $L$ from $\eta$ and $\Delta$ on $Q_B$.
Then $(L,\Delta,\eta)$ with the induced grading forms a charge factor on $F$.
Let $I$ be an weak GVA-ideal of $F$ generated by the above relations.
Since the above relations are $L$-homogeneous, $I$ is a GVA-ideal.
Consider the quotient GVA $W=F/I$.
Since $V_L$ satisfies the above relations, we have the projection $W\rightarrow V_L$.
We show $W\cong V_L$.

Set $\mathfrak{h}=L\otimes_{\mathbb{Z}} \mathbb{C}$ and bilinearly extend the form $(\cdot|\cdot)$ to $\mathfrak{h}$.
Consider the abelian Lie algebra structure on $\mathfrak{h}$.
Consider the corresponding affine Lie algebra (Heisenberg algebra)
$\hat{\mathfrak{h}}=\mathfrak{h}\otimes \mathbb{C}[t,t^{-1}]\oplus \mathbb{C}c$
with $[x\otimes t^m,y\otimes t^n]=(x|y)m\delta_{m+n,0}c$ for $x,y\in\mathfrak{h}$, $m,n\in\mathbb{Z}$
and $[c,\mathfrak{h}]=0$.
We denote $x\otimes t^m=x(m)$.

\begin{lem}
The Heisenberg algebra $\hat{\mathfrak{h}}$ acts on $W$ by $c v=v$ and $a(m) v=\widetilde{a}(m) v$ for $a\in \Pi\subset \mathfrak{h}$, $m\in\mathbb{C}$ and $v\in W$.
\end{lem}

\begin{proof}
Let $a,b$ be elements of $\Pi$ and $m,k$ be complex numbers. Let $v$ be an element of $U$.
By the Borcherds identity with $n=0$ and $c=v$, we have
$a(m) b(k) v-b(k) a(m) v=m(a|b)\delta_{m+k,0}v$.
Thus, we have the lemma.
\end{proof}

Consider a total order $\geq$ on $\Pi$.
Let $\alpha$ be an element of $L$.
Then $\alpha$ has the form $\alpha=\sum_{i=1}^n s_i b_i$ with elements $b_1,\ldots,b_n\in \Pi$ with
$b_1\leq \cdots \leq b_n$ and signatures $s_1,\ldots,s_n\in \{\pm 1\}$ with $s_i=s_j$ if $b_i=b_j$.
Set
\[
v^\alpha=\Psi(v^{s_1 b_1},-1;\cdots;v^{s_n b_n},-1)=\prod_{i=1}^n \left( v^{s_i b_i}(-1-(b_i|{\textstyle \sum_{j=1}^{i-1}}s_j b_j))\right) |0\rangle.
\]
It belongs to the set $\mathcal{C}(W,\pm \Pi, \mathcal{N})$.
Put $\mathcal{C}=\mathcal{C}(W,\pm \Pi, \mathcal{N})$.
For a monomial $v=v^{a_1}(n_1)\cdots v^{a_m}(n_m)|0\rangle$ with $a_i\in \pm \Pi$ and $n_i\in \mathbb{C}$,
we set $\gamma(v)=a_1+\cdots+a_m\in L$.
Set $\mathcal{C}^\alpha=\{v\in \mathcal{C}|\gamma(v)=\alpha\}$.

\begin{lem}
\begin{enumerate}
\item The subspace spanned by the set $\mathcal{C}^\alpha$ coincides with the submodule $U(\hat{\mathfrak{h}})v^\alpha$.
\item $v^\alpha$ is a highest weight vector of $\hat{\mathfrak{h}}$ with $h(0) v^\alpha= (h|\alpha)v^\alpha$ for $h\in\mathfrak{h}$.
\end{enumerate}
\end{lem}

\begin{proof}
By the Borcherds identity with $n=0$, we have 
$\widetilde{a}(m) v^b(k)-v^b(k) \widetilde{a}(m)=(a|b)v^b$.
By the translation axiom of GVA and the relation $Tv^a=\widetilde{a}(-1) v^a$, 
we have the lemma.
\end{proof}

\begin{proof}[Proof of Lemma \ref{sec:lemlattice}.]
Note that $\hat{\mathfrak{h}}$ acts on $V_L=\bigoplus_{\alpha\in L} M(1)\otimes e^\alpha$.
By the above lemma, $W=\bigoplus_{\alpha\in L} U(\hat{\mathfrak{h}})v^\alpha$, since $I$ is $L$-graded.
Since $v^\alpha$ is a highest weight vector with weight $h(0) v^\alpha=(h|\alpha)v^\alpha$, and 
the projection $W\rightarrow V_L$ is a surjective $\hat{\mathfrak{h}}$-module homomorphism, we see that the projection is an isomorphism,
which completes the proof.
\end{proof}

\begin{thm}
The lattice generalized vertex algebra $V_L$ is presented by generators $\{v^a|a\in \pm \Pi\}$ with locality bound $\mathcal{N}(a,b)=-(a|b)$ and $\eta(a,b)=e^{\pi i (a|b)}\omega(a,b)$ ($a,b\in\pm \Pi$) and relations $\{v^a((a|a)-1)v^{-a}=\varepsilon(a,-a)|0\rangle |a\in \Pi\}$.
\end{thm}

\begin{proof}
Consider a set of generators $B=\{X^a|a\in \pm\Pi\}$ with the symmetric map $\mathcal{N}(X^a,X^b)=-(a|b)$ and bimultiplicative map
$\eta(X^a,X^b)=e^{\pi i (a|b)}\omega(a,b)$ ($a,b\in\pm \Pi$).
Consider the corresponding free generalized vertex algebra $F=F_{\mathcal{N},\eta}(B)$.
Then by the universality of $F$, we have the projection $\pi:F\rightarrow V_L$ with $X^a\mapsto v^a$.
For $a\in \pm \Pi$ denote $K^a=X^a((a|a)-1) X^{-a}$, $H^a=X^a((a|a)-2)X^{-a}$.
By eq.\ (\ref{eqn:local2}), we have
$
H^a(0) H^b=(a|b)K^a(-2) K^b,
$
$H^a(1) H^b=(a|b)K^a(-1)K^b$,
$H^a(0) X^b=(a|b)K^a(-1)X^b$,
$H^a(-1)X^a=X^a(-2)K^a+(a|a)K^a(-2)X^a$.
Let $J\subset F$ denote the ideal generated by the relations $K^a=\varepsilon(a,-a)|0\rangle$ ($a\in \Pi$).
Then by the above equalities and Lemma \ref{sec:lemlattice}, $F/J\cong V_L$, which completes the proof.
\end{proof}

This is the generalization of the result of \cite{R}.

\section{The properties of the generalized principal subspaces}

Since the free generalized vertex algebras are isomorphic to generalized principal subspaces,
we study the generalized principal subspaces.

Let $l$ be a positive integer.
Let $\mathfrak{h}$ be a $l$-dimensional vector space over $\mathbb{C}$ equipped with a bilinear form
$(\cdot|\cdot): \mathfrak{h}\times \mathfrak{h} \rightarrow \mathbb{C}$.
Let $\varepsilon: \mathfrak{h}\times \mathfrak{h}\rightarrow \mathbb{C}^\times$
be a $2$-cocycle.
Let $L\subset \mathfrak{h}$ be a lattice equipped with 
a $\mathbb{Z}$-basis $B\subset \mathfrak{h}$.
Assume that $B$ is also a $\mathbb{C}$-basis of $\mathfrak{h}$, so that
$\mathfrak{h}\cong \mathbb{C}\otimes_{\mathbb{Z}}L$.
Consider the ($\varepsilon$-modified) generalized vertex algebra
$
V_\mathfrak{h}=M(1) \otimes_{\mathbb{C}} \mathbb{C}[\mathfrak{h}]
$
associated to the vector space $\mathfrak{h}$.
Let $\lambda$ be an element of $L^\circ$.
Consider the generalized principal subspaces $W_L=W_L(B)$ and $W_L(\lambda)=W_L(B,\lambda)$.

\subsection{Graded dimensions of the generalized principal subspaces.}

Recall the charge-gradings and weight-gradings of the generalized principal subspaces.
Note that they are compatible.

\begin{dfn}
The {\it graded dimension} $\chi_{W_L(\lambda)}$ of $W_L(\lambda)$ is
\[
\chi_{W_L(\lambda)}(\bm{x};q)= \sum_{\alpha \in \mathfrak{h},n\in \mathbb{C}} \dim_\mathbb{C} \left( \left(W_L(\lambda)^\alpha\right)_n\right) q^n x_1^{i_1}\cdots x_l^{i_l},
\]
where $i_1\beta_1+\cdots+i_l\beta_l=\alpha$.
\end{dfn}

To compute the graded dimension, consider the symbols
\[
(q)_k=(q; q)_k=(1-q)\cdots(1-q^k), \ \ \  (k\geq 1),
\]
$(q)_0=1$,
and
$(q)_{\infty}=\prod_{i=1}^\infty (1-q^i)$.
Here $
(a;q)_k=(1-a)(1-aq)\cdots (1-aq^{k-1})
$
is the $q$-Pochhammer symbol.
Recall that $1/(q)_k$ agrees with the generating function of the partitions into parts not greater than $k$,
 therefore agrees with the generating function of the partitions into at most $k$ parts.
Recall further that
\begin{equation*}
\frac{1}{(q)_{\infty}} = \frac{1}{\varphi(q)} = \sum_{k=0} ^{\infty} p(k) q^k.
\end{equation*}
Here, $\varphi(q)$ is the Euler function and $p(k)$ is the number of the un-restricted partitions of an integer $k$.
Consider the Gram matrix $A=\left(( \beta_i | \beta_j )\right)_{i,j}$.
Since $\lambda\in \mathfrak{h}$, $\lambda$ has the form $j_1\beta_1+\cdots +j_l\beta_l$ with $j_1,\ldots, j_l \in \mathbb{C}$.
Put $\bm{j}=(j_1,\ldots,j_l)$.

\begin{thm} \label{sec:thmcharacter}
The graded dimension of $W_L(\lambda)$ is given by
\[
\chi _{W_L(\lambda) } (\bm{x};q) 
= \sum_{i_1,\ldots,i_l \geq 0}
\frac{q^{ \frac{ (\bm{i}+\bm{j}) \cdot A \cdot (\bm{i}+\bm{j}) } {2} } } {  (q)_{i_1} \cdots (q)_{i_l}} 
x_1^{i_1} \cdots x_l^{i_l},
\]
where $\bm{i}=(i_1,\ldots,i_l)$.
\end{thm}

\begin{proof}
The assertion follows from  Corollary \ref{sec:corbasis2}
and the relation
\[
{\rm wt}(e^{\alpha+\lambda})=\frac{ ( \alpha+\lambda | \alpha+\lambda )} {2} = \frac{ (\bm{i}+\bm{j}) \cdot A \cdot (\bm{i}+\bm{j}) } {2},
\]
where $\alpha=i_1\beta_1+\cdots +i_l\beta_l$.
\end{proof}

\subsection{Duality of the principal subspaces.}

We prove the duality theorem of the principal subspaces.
Assume that the bilinear form $(\cdot|\cdot)$ is non-degenerate.
Then there exists $\beta^\circ_1,\ldots,\beta^\circ_l \in \mathfrak{h}$ such that
\[
(\beta^\circ_i|\beta_j)=\delta_{ij},
\]
where $\delta_{ij}$ is the Dirac delta function.
Consider the dual lattice $L^\circ$ with the dual $\mathbb{Z}$-basis $B^\circ=\{ \beta^\circ_1,\ldots,\beta^\circ_l \}$.

Now, we define the notion of the commuteness of generalized vertex algebras.
Let $V$ be a generalized vertex algebra and $S$ a subset of $V$.
\begin{dfn}
The {\it commutant} of $S$ in $V$ is the vector space
\[
\mathrm{Com} (S,V)=\{ v\in V | v(n)s=0 \,\,\, \mbox{for} \,\, s\in S \,\, \mbox{and} \,\, n >-1 \}.
\]
\end{dfn}

Furthermore, we define the notion of the invariant subspaces.
Let $A$ be a vector subspace of a module of $V$.
\begin{dfn}(cf. \cite{LiL})
The {\it invariant subspace} of $S$ inside $A$ is 
\[
A^{S_+}=\{ v\in A| s(n)v=0 \,\,\, \mbox{for} \,\, s\in S \,\, \mbox{and} \,\, n >-1 \}.
\]
\end{dfn}
When $A=V$, the vector space $V^{S_+}$ agrees with $\mathrm{Com} (S,V)$.

The following theorem is our main result.

\begin{thm}\label{sec:thmdual1}
The invariant subspace $(V_L)^{(W_{L^\circ})_+}$ of the generalized principal subspace $W_{L^\circ}$ inside $V_L$ coincides with the generalized principal subspace $W_L$.
\end{thm}

Note that, since $(L^\circ)^\circ=L$, we also obtain $(V_{L^\circ})^{(W_L)_+}=W_{L^\circ}$ from the theorem.
To prove the theorem, we use what we call {\it Feigin-Stoyanovsky type filtration}.

Let $\mu$ and $\nu$ be elements of $L$.
Then, $\mu$ and $\nu$ have the form $\mu=i_1\beta_1+\cdots+i_l\beta_l$ and $\nu=j_1\beta_1+\cdots+j_l\beta_l$ 
with $i_n, j_n \in \mathbb{Z}$ ($n=1,\ldots,l$).
We say
\[
\mu \preceq \nu,
\]
if $i_n \geq j_n$ for $n=1,\ldots,l$.
Then $(L,\preceq)$ form a filtered set.
When $\mu\preceq \nu$, we have $W_L(\mu) \subset W_L(\nu)$.
Set $F_L=\{W_L(\mu)|\mu \in L\}$.
Consider the inductive system $(F_L, \subset)$.
Then, the inductive limit $\displaystyle \lim_{\longrightarrow} F_L$ coincides with the whole lattice vertex algebra $V_L$,
since the graded dimensions of the elements of $F_L$ converge to the graded dimension of $V_L$ by Theorem \ref{sec:thmcharacter}.
We call the inductive system {\it Feigin-Stoyanovsky type filtration}.
By Corollary \ref{sec:corbasis2}, we know a basis of $W_L(\mu)$.
Consider the total order $\beta_1\leq \cdots \leq \beta_l$ on $B$ and locality bound $\mathcal{N}(\alpha,\beta)=-(\alpha|\beta)$ on $B$.
We denote $\mathcal{C}(W_L,B,\mathcal{N};\mu)=\mathcal{C}^{(\mu)}$.
Now, consider the subset $\mathcal{C}^{(\mu)}_1$ consisting of the elements
\[
\Psi(e^{\mu},-1;e^{a_1},m_1;\cdots;e^{a_k},m_k)=\prod_{i=1}^k \left( e^{a_i}(m_i-(a_i|\mu+{\textstyle \sum_{j=1}^{i-1}} a_j))\right) e^\mu
\]
with a non-negative integer $k$, elements $a_1,\ldots,a_k\in B$ with $a_1\leq \cdots\leq a_n$ and negative integers
$m_1,\ldots,m_k$ with $m_1\leq -2$ if $a_1=\beta_1$ and $m_i\geq m_j$ if $i<j$ and $a_i=a_j$.

Then, $\mathcal{C}^{(\mu)}=\mathcal{C}^{(\mu)}_1 \sqcup (\varepsilon(\mu,\beta_1)^{-1} \mathcal{C}^{(\mu+\beta_1)}$),
since $e^{\beta_1}(-1-(\beta_1|\mu))e^{\mu}=\varepsilon(\mu,\beta_1)e^{\mu+\beta_1}$.
Here, $cX$ denotes the set $\{cx|x\in X\}$ for a non-zero complex number $c$ and a set $X$.
Set $W_L(\mu)'=\mathbb{C}\mbox{-span} (\mathcal{C}^{(\mu)}_1)$.
Then, $W_L(\mu)=W_L(\mu)' \oplus W_L(\mu+\beta_1)$.

\begin{lem}\label{sec:lemmap}
The linear map
\begin{equation}\label{eqn:inj}
(e^{\beta^\circ_1})(-(\beta^\circ_1|\mu)-1): W_L(\mu)' \longrightarrow V_{\mathfrak{h}}
\end{equation}
is injective. The linear map
\begin{equation}\label{eqn:zero}
(e^{\beta^\circ_1})(-(\beta^\circ_1|\mu)-1): W_L(\mu+\beta_1) \longrightarrow V_{\mathfrak{h}}
\end{equation}
is zero.
\end{lem}

\begin{proof}
Let $v$ be an element of the basis $\mathcal{C}^{(\mu)}$.
Then $v$ has the form
\[
v=\Psi(e^\mu,-1;e^{a_1},m_1;\cdots;e^{a_k},m_k)=\prod_{i=1}^k \left( e^{a_i}(m_i-(a_i|\mu+{\textstyle \sum_{j=1}^{i-1}}a_j))\right)e^{\mu}
\]
with a non-negative integer $k$, elements $a_1,\ldots,a_k\in B$
with $a_1\leq \cdots \leq a_k$ and 
negative integers $m_1,\ldots,m_k$ with $m_i\geq m_j$ if $i<j$ and $a_i=a_j$.
By commutation relation (\ref{eqn:commutation}) and
the equality $(e^{\beta_1^\circ})(-(\beta_1^\circ|\mu)-1)e^\mu=\varepsilon(\beta_1^\circ,\mu)e^{\beta_1^\circ+\mu}$, we have
\begin{eqnarray*}
&&(e^{\beta_1^\circ})(-(\beta_1^\circ|\mu)-1) v\\
&&\quad =c\cdot \prod_{i=1}^k \left( e^{a_i}(m_i-(a_i|\mu+{\textstyle \sum_{j=1}^{i-1}}a_j)) \right) e^{\beta_1^\circ+\mu}\\
&&\quad =c\cdot \Psi\left(e^{\beta_1^\circ+\mu},-1;e^{a_1},m_1+(a_1|\beta_1^\circ);\cdots;e^{a_k},m_k+(a_k|\beta_1^\circ)\right)
\end{eqnarray*}
with a non-zero scalar $c\in \mathbb{C}^\times$.
When $a_1=\beta_1$ and $m_1=-1$, we have $(e^{\beta^\circ_1})(-(\beta^\circ_1|\mu)-1)v=0$, since 
$(e^{\beta_1})(-1-(\beta_1|\mu))e^{\beta^\circ_1+\mu}=0$.
Hence the map (\ref{eqn:zero}) is zero.
When $a_1>\beta_1$ or $m_1<-1$, we see that $c^{-1}(e^{\beta^\circ_1})(-(\beta^\circ_1|\mu)-1)v$ belongs to
$\mathcal{C}^{(\mu+\beta^\circ_1)}$.
Moreover, if $w\in \mathcal{C}^{(\mu)}$ satisfies $w\not =v$, then any non-zero multiple of $(e^{\beta^\circ_1})(-(\beta^\circ_1|\mu)-1)w$
is not equal to $c^{-1} (e^{\beta^\circ_1})(-(\beta^\circ_1|\mu)-1)v$ in $\mathcal{C}^{(\mu+\beta^\circ_1)}$.
By linearly independence of the set $\mathcal{C}^{(\mu+\beta^\circ_1)}$, the map (\ref{eqn:inj}) is 
injective.
\end{proof}

\begin{proof}[Proof of Theorem \ref{sec:thmdual1}]
Put $X=(V_L)^{(W_{L^\circ})_+}$.
By the commutation relation (\ref{eqn:commutation}), we obtain $W_L \subset X$, since $e^{\beta_i}(m) |0\rangle=0$ for $m\geq 0$ and $i=1,\ldots,l$.
We prove $W_L \supset X$.
Let $v$ be an element of $X=(V_L)^{(W_{L^\circ})_+}$.
Since $F_L$ is a filtration, there exists $\mu \in L$ such that
$v \in W_L(\mu)$.
Since $\mu \in L$, $\mu$ has the form $\mu=i_1\beta_1+\cdots+i_l\beta_l$ with $i_1,\ldots,i_l \in \mathbb{Z}$.
We show
\begin{equation}\label{eqn:reduction}
v\in W_L(i_2\beta_2+\cdots+i_l\beta_l).
\end{equation}
When $i_1 \geq 0$, we have $W_L(\mu) \subset W_L(i_2\beta_2+\cdots+i_l\beta_l)$.
Therefore, in this case, (\ref{eqn:reduction}) holds.
Assume that $i_1\leq -1$.
Then, since $-(\beta^\circ_1|\mu)-1=-i_1-1$ is greater than or equal to $0$ and $v$ belongs to $X$, we have $(e^{\beta^\circ_1})(-(\beta^\circ_1|\mu)-1)v=0$.
By Lemma \ref{sec:lemmap}, $v\in W_L(\mu+\beta_1)$.
Repeating the procedure, we have $v\in W_L(i_2\beta_2+\cdots+i_l\beta_l)$.
Hence we have (\ref{eqn:reduction}).
Since the ordering $(\beta_1,\ldots,\beta_l)$ of $B$ is arbitrary, we also have $v\in W_L(i_3\beta_3+\cdots+i_l\beta_l)$, and eventually we obtain
$v\in W_L(0)=W_L$, which completes the proof.
\end{proof}

Now, we give an application of the theorem.

Let $\mathfrak{g}$ be a rank $l$ simply laced simple Lie algebra.
Let $\mathfrak{h}$ be a Cartan subalgebra of $\mathfrak{g}$ with the root lattice $Q$ and weight lattice $P$.
Let $\alpha_1,\ldots,\alpha_l$ be simple roots of $\mathfrak{g}$
with the Borel nilradical $\mathfrak{n}_{\pm}$
and the fundamental weights $\lambda_1,\ldots,\lambda_l$.

Consider the affine Kac-Moody Lie algebra $\mathfrak{g}^{(1)}=\mathfrak{g}[t,t^{-1}]\oplus \mathbb{C}K \oplus \mathbb{C}D$.
Set the fundamental weights $\Lambda_0,\Lambda_1,\ldots,\Lambda_l$ of $\mathfrak{g}^{(1)}$.
Let $L(\Lambda_n)$ be the level one fundamental representations.
Then, $L(\Lambda_0)$ has naturally a vertex algebra structure and $L(\Lambda_n)$ a structure of module over $L(\Lambda_0)$
for $n=0,\ldots,l$ which are compatible with the action of $\mathfrak{g}^{(1)}$.

Let $V_Q$ and $V_P$ be the (generalized) vertex algebra associated to the lattices $Q$ and $P$.
Then, $V_Q$ is a vertex subalgebra of $V_P$, and $V_P$ has the structure of $V_Q$-module.
Note that the irreducible modules $V_{Q+\lambda_n}$ over $V_Q$ ($n=1,\ldots,l$) are submodules of $V_P$.

It is well-known that the vertex algebra $L(\Lambda_0)$ is isomorphic to $V_Q$ and the module $L(\Lambda_n)$ is isomorphic to
$V_{Q+\lambda_n}$ ($n=1,\ldots,l$).

Consider the generalized principal subspace $W_P$ of $V_P$.
Let $\lambda$ be an element of $P$.
We denote
\[
(W_P)^{\overline{\lambda}}=\bigoplus_{\beta\in Q} (W_P)^{\lambda+\beta}.
\]

\begin{cor}
Under the isomorphism, the invariant space $L(\Lambda_0)^{(\mathfrak{n}_+[t])}$ agrees with $(W_P)^{\overline{0}}$.
Moreover, the invariant space $L(\Lambda_n)^{(\mathfrak{n}_+[t])}$ agrees with $(W_P)^{\overline{\lambda_n}}$ for $n=1,\ldots,l$.
\end{cor}

\begin{cor}\label{sec:corcommutant}
The commutant $\mathrm{Com}(W_Q,V_Q)$ of the Feigin-Stoyanovsky principal subspace $W_Q$ inside $V_Q$ agrees with $(W_P)^{\overline{0}}$.
\end{cor}

\begin{proof}[Proof of the corollaries]
We see that $W_P\cap V_{Q+\lambda}= (W_P)^{\overline{\lambda}}$ for $\lambda\in P$.
Since $P$ is the dual lattice of $Q$, by Theorem \ref{sec:thmdual1}, we have the corollaries.
\end{proof}

\subsection*{Acknowledgments}
The author wishes to express his thanks to his advisor, Professor Atsushi Matsuo for helpful advice and kind encouragement.
He also wishes to express his thanks to Toshiyuki Abe, Hiroshi Yamauchi and Masanari Okumura for helpful discussions.
This work was supported by JSPS KAKENHI Grant Number 14J09236.

\end{document}